\documentclass{amsart}

\newtheorem{theorem}{Theorem}[section]
\newtheorem{lemma}[theorem]{Lemma}
\newtheorem{proposition}[theorem]{Proposition}

\theoremstyle{definition}
\newtheorem{definition}[theorem]{Definition}

\newtheorem{assumption}[theorem]{Assumption}
\newtheorem*{rep@theorem}{\rep@title}
\newcommand{\newreptheorem}[2]{%
\newenvironment{rep#1}[1]{%
 \def\rep@title{#2 \ref{##1}}%
 \begin{rep@theorem}}%
 {\end{rep@theorem}}}

\newreptheorem{theorem}{Theorem}

\theoremstyle{remark}
\newtheorem{remark}[theorem]{Remark}

\numberwithin{equation}{section}

\usepackage{cite}
\usepackage{graphicx}
\usepackage{pgfplots,tikz}
\usetikzlibrary{patterns}
\usepackage{appendix}
\usepackage[hidelinks]{hyperref}
\hypersetup{
    colorlinks=true,
    citecolor=blue,
    linkcolor=blue,
}
\usepackage{enumitem}

\newcommand{\setr}{\mathbb{R}}

\newcommand{\id}{\operatorname{Id}}

\newcommand{\bigo}{\mathcal{O}}

\newcommand{\fsccinf}{C_c^\infty}

\newcommand{\sgn}{\operatorname{sgn}}

\newcommand{\pd}{\partial}

\newcommand{\supp}{\operatorname{supp}}

\newcommand{\cim}{\operatorname{Im}}
\newcommand{\cre}{\operatorname{Re}}

\newcommand{\abs}[1]{\lvert#1\rvert}

\begin{document}

\title[Sharp Polynomial Decay in Cylindrical Waveguides]{Sharp Polynomial Decay for Waves Damped from the Boundary in Cylindrical Waveguides}

\author{Ruoyu P. T. Wang}
\address{Department of Mathematics, Northwestern University, Evanston, Illinois 60208}
\email{rptwang@math.northwestern.edu}

%    General info

\date{\today}


\begin{abstract}
We study the decay of global energy for the wave equation with Hölder continuous damping placed on the $C^{1,1}$-boundary of compact and non-compact waveguides with star-shaped cross-sections. We show there is sharp $t^{-1/2}$-decay when the damping is uniformly bounded from below on the cylindrical wall of product cylinders where the Geometric Control Condition is violated. On non-product cylinders, we also show that there is $t^{-1/3}$-decay when the damping is uniformly bounded from below on the cylindrical wall. 
\end{abstract}

\maketitle
\section{Introduction}
%\subsection{Introduction}
\subsection{Introduction and Main Results}
In this paper, we are interested in obtaining polynomial decay of global energy for waves that are damped from the boundary of cylindrical waveguides, by using the semigroup theory and the resolvent estimates for interior impedance problems. 
\subsubsection{Sharp Polynomial Decay of Damped Waves}
Let $d\ge 1$ and we impose a geometric assumption:
\begin{definition}[Product Cylinders]\label{0gp}A compact product cylinder $\Omega$ with star-shaped cross-sections is a $(d+1)$-dimensional product manifold $\Omega=\Omega_x\times[-1,1]_y$, where $\Omega_x\subset\mathbb{R}^d$ is bounded and has $C^{1,1}$-boundary $\pd\Omega_x$ which has finitely many components. Assume further that $\Omega_x$ is star-shaped with respect to $x=0$, that is, there exists $c_0>0$ with $x\cdot n_x\ge c_0>0$ almost everywhere on $\pd\Omega_x$, where $n_x$ is the outward-pointing normal defined almost everywhere. We call $\Omega$ an infinite product cylinder if we replace $[0,1]_y$ by $\mathbb{R}_y$ in the assumption of compact ones. See Figure \ref{f1} and Figure \ref{f3} for examples. 
\end{definition}
A compact product cylinder $\Omega$ is a bounded Lipschitz domain: even when $\pd\Omega_x$ is smooth, $\Omega$ is still a manifold with corners of codimension 2. The boundary of $\Omega$ consists of three $C^{1,1}$-pieces: we have $\pd\Omega=\Gamma\cup(\Gamma_-\sqcup\Gamma_+)$ where
\begin{equation}\label{1l12}
\Gamma=\{(x,y)\in\Omega: x\in \pd\Omega_x\}, \ \Gamma_\pm=\{(x,\pm1)\in\Omega\}.
\end{equation}
We call $\Gamma$ the cylindrical wall and $\Gamma_\pm$ the cylindrical caps. We will consider two non-negative real $L^\infty$-boundary damping functions:
\begin{assumption}[Damping Assumption]\label{0ad} Let the damping functions $a(p), b(p)$ in $L^\infty(\Gamma)$ satisfy
\begin{gather}
0<a_0\le a(p)\le a_1, \ p\in\Gamma,\\
0\le b(p)\le b_1, \ p\in\Gamma
\end{gather}
for some constants $a_0, b_1, a_1>0$. Note that we have
\begin{equation}\label{1l4}
b(p)\le\frac{b_1}{a_1}a(p), \ p\in \Gamma.
\end{equation}
\end{assumption}
\begin{figure}
\begin{tikzpicture}[minimum size=0.01cm]
\draw[fill=gray!50] (0, 0.8) arc (90:270:-0.16 and 0.8) -- (3, -0.8) -- (3, -0.8) arc (270:90:-0.16 and 0.8) -- (3,0.8) -- cycle;
\draw (0,0) ellipse (0.16 and 0.8);
\node at (-0.4,0) {$\Gamma_-$};
\draw (0, -0.8) -- (3, -0.8);

\draw (3, 0.8) arc (90:270:-0.16 and 0.8);
\draw [dashed] (3,0) ellipse (0.16 and 0.8);
\node at (3.5,0) {$\Gamma_+$};
\draw (0,0.8) -- (3,0.8); 

\draw (1.5, 0.8) arc (90:270:-0.16 and 0.8);
\draw [dashed] (1.5,0) ellipse (0.16 and 0.8);
\fill [pattern=north east lines] (1.5,0) ellipse (0.16 and 0.8);
\node at (1.9,0) {$\Omega_x$};

\node at (1.5,1) {$\Gamma$}; 

\draw [->] (0.5,-1) -- (2.5,-1);
\node at (1.5, -1.2) {$y$};
\end{tikzpicture}\hspace{5em}
\begin{tikzpicture}[minimum size=0.01cm]
\draw[fill=gray!50] (0, 0.75) -- (1.6, 0.75) -- (1.6, 0.85) -- (0, 0.85) -- cycle;
\draw[fill=gray!50] (0, -0.75) -- (1.6, -0.75) -- (1.6, -0.85) -- (0, -0.85) -- cycle;
\draw (0, -0.8) -- (0, 0.8);
\draw (0, -0.8) -- (1.6, -0.8);
\draw (1.6, -0.8) -- (1.6, 0.8);
\draw (1.6, 0.8) -- (0, 0.8);
\node at (-0.3, 0) {$\Gamma_-$};
\node at (1.9, 0) {$\Gamma_+$};
\node at (0.8, 1.05) {$\Gamma$};
\draw [->] (0.15,-1) -- (1.45,-1);
\node at (0.8, -1.2) {$y$};
\draw (0.8, 0.8) -- (0.8, -0.8);
\node at (1.1, 0) {$\Omega_x$};
\end{tikzpicture}
\caption{Two examples of compact product cylinders. The support of the damping functions $a(p)$ is shaded grey. The cross-sections $\Omega_x$ need to be star-shaped.} 
\label{f1}
\end{figure}
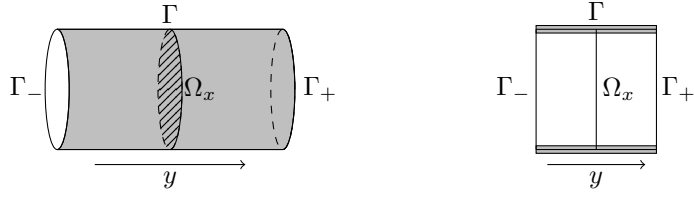
Let $\Delta\ge 0$ be the positive Laplacian on $\Omega$. Consider the damped wave equation with the impedance boundary condition on $\Gamma$ and the Neumann boundary condition on $\Gamma_\pm$:
\begin{gather}\label{1l1}
\left(\pd_t^2+\Delta\right)u(t,p)=0, \ (t, p)\in \setr_t\times \Omega,\\
\left(\pd_n +a(p)\pd_t\gamma+b(p)\gamma\right)u(t,p)=0, (t, p)\in \setr_t\times \Gamma,\\
\pd_n u(t,p)=0, (t, p)\in \setr_t\times \Gamma_\pm,\\
u(0, p)=u_0(p)\in H^2(\Omega), \ \pd_t u(0, p)=u_0(p)\in H^1(\Omega),
\end{gather}
where $\gamma:H^1(\Omega)\rightarrow H^{1/2}(\pd\Omega)$ is the trace operator and $\pd_n$ is the conormal derivative with respect to the outward-pointing normal to $\pd\Omega$ defined almost everywhere. We drop $\gamma$ as long as there is no confusion. Define the energy of the solution $u$ to the system \eqref{1l1} by
\begin{equation}
E(u,t)=\frac12 \int_\Omega \abs{\nabla u}^2+\abs{\pd_t u}^2\ dxdy.
\end{equation}
We want to understand at which rates $f(t)$ is the damped wave equation stable, that is,
\begin{equation}
E(u,t)^{\frac12}\le Cf(t)(\|u_0\|_{H^2(\Omega)}+\|u_1\|_{H^1(\Omega)})
\end{equation}
for some $C, f$ independent of $u_0$ and $u_1$. 

The damping assumption \ref{0ad} requires damping to fully vanish on $\Gamma_\pm$ and thus violates the Geometric Control Condition formulated in \cite{blr92}. This is the dynamical control condition requiring that every generalised bicharacteristic hits the boundary in finite time at a non-diffractive point where the damping is not trivial. It is known in \cite{blr92} that on manifolds with smooth boundary, the Geometric Control Condition is equivalent to the exponential energy decay, and thus we can not expect such rates in our case. On another hand, as long as the damping is not trivial somewhere on the smooth boundary, Lebeau and Robbiano showed in \cite{lr97} that there is logarithmic decay using the Carleman estimates. 

We are then interested in obtaining intermediate polynomial energy decay in cylindrical domains. Phung showed in \cite{phu08} that continuous damping on the boundary of 3D cylindrical waveguides with cross-sections of $C^2$-boundary gives a $t^{-\delta}$-polynomial decay for some $\delta>0$ without specifying any upper bound for $\delta$. Nishiyama proved in \cite{nis13} that $C^{1,1}$-damping gives a $t^{-1/3}$-energy decay on 2D $C^{1,1}$-rectangular domains with star-shaped cross-sections. We aim to sharpen those rates, and weaken the regularity assumption to $C^{0,1/2+\delta}$-damping.
\begin{theorem}\label{0t1}
Let $u$ be the unique solution to the damped wave equation \eqref{1l1} on a compact product cylinder $\Omega$ defined in Definition \ref{0gp}, impose the damping assumption \ref{0ad} and assume $a(p), b(p)\in C^{0,1/2+\delta}(\Gamma)$ for some $\delta>0$. Then there is $C>0$ such that
\begin{equation}
E(u,t)^{\frac12}\le Ct^{-\frac12}\left(\|u_0\|_{H^2(\Omega)}+\|u_1\|_{H^1(\Omega)}\right),
\end{equation}
where $C$ is independent of the initial data $u_0$ and $u_1$. In other words, the damped wave equation \eqref{1l1} is stable at the rate $t^{-1/2}$.  
\end{theorem}
\begin{remark}
\begin{enumerate}[wide]
\item The rate $t^{-1/2}$ is sharp: see Proposition \ref{3t5}. 

\item One could replace the Neumann boundary conditions on $\Gamma_\pm$ by the Dirichlet boundary condition or the periodic boundary condition
\begin{gather}\label{1l18}
u(t,p)=0, \ (t, p)\in \setr_t\times \Gamma_\pm; \textnormal{or}\\
\label{1l17}
u(t,x,-1)=u(t,x,1), \ \pd_y u(t,x,-1)=\pd_y u(t,x,1), \ (t, x)\in \setr_t\times \Omega_x,
\end{gather}
and the same result holds. See remarks \ref{0t7}, \ref{2t7}, \ref{2t8}, \ref{3t6}, \ref{3t4}(3) for details. 

\item The Hölder regularity of $a,b$ is needed to guarantee that the solutions to the damped wave equation are in $H^2(\mathbb{R}_{t\ge 0}\times \Omega)$. We use a monotonicity argument to improve the regularity of $a, b$ by one order without tarnishing the optimality of the polynomial decay. This is surprising, since in \cite{aln14,dk20} we know that sharp polynomial decay depends on the Hölder regularity of interior damping. See Remark \ref{3t4}(1).

\item The same result is expected to hold if one replaces $[-1,1]_y$ by any compact $C^\infty$-manifolds, either without boundary, or with the homogeneous Dirichlet or Neumann boundary conditions on the $C^\infty$-boundary. To keep the paper concise, we chose not to develop those issues here, but the proof should be similar. 
\end{enumerate}
\end{remark}

\begin{figure}
\begin{tikzpicture}[minimum size=0.01cm]

\draw (0,0) ellipse (0.16 and 0.8);
\draw (0, -0.8) -- (3, -0.8);

\draw (3, 0.8) arc (90:270:-0.16 and 0.8);
\draw [dashed] (3,0) ellipse (0.16 and 0.8);

\draw (0,0.8) -- (3,0.8); 

\draw [->](1, 0) node [circle, minimum size=0.5mm, inner sep=0, draw=black, fill=black] {} -- (3, 0);

\draw [dashed, ->](1, 0) node [circle, minimum size=0.5mm, inner sep=0, draw=black, fill=black] {} -- (3, 0.3);

\draw [dashed, ->](1, 0) node [circle, minimum size=0.5mm, inner sep=0, draw=black, fill=black] {} -- (3, 0.6);

\draw [dashed, ->](1, 0) node [circle, minimum size=0.5mm, inner sep=0, draw=black, fill=black] {} -- (3, -0.3);

\draw [dashed, ->](1, 0) node [circle, minimum size=0.5mm, inner sep=0, draw=black, fill=black] {} -- (3, -0.6);

\node at (1.5,1) {$\Gamma$}; 
\end{tikzpicture}\hspace{5em}
\begin{tikzpicture}[minimum size=0.01cm]

\draw (0,0) ellipse (0.16 and 0.8);
\draw (0, -0.8) -- (3, -0.8);

\draw (3, 0.8) arc (90:270:-0.16 and 0.8);
\draw [dashed] (3,0) ellipse (0.16 and 0.8);

\draw (0,0.8) -- (3,0.8); 

\draw [->](1, 0) node [circle, minimum size=0.5mm, inner sep=0, draw=black, fill=black] {} -- (3, 0);

\draw [dashed](1, 0) node [circle, minimum size=0.5mm, inner sep=0, draw=black, fill=black] {} -- (3, 0.24);

\draw [dashed](3, 0.24) -- (0, 0.6);

\draw [dashed, ->](0, 0.6) -- (5/3, 0.8) node [circle, minimum size=0.5mm, inner sep=0, draw=black, fill=black] {};

\node at (1.5,1) {$\Gamma$}; 
\end{tikzpicture}
\caption{Non-concentration along geodesics of momentum purely in $y$. In the figure on the left, the high frequency waves fail to concentrate exponentially near the solid geodesic of momentum purely in $y$ which does not hit the dissipated boundary. In the figure on the right, those dashed geodesics will eventually hit dissipative boundary $\Gamma$, but are damped very weakly due to small angles of incidence.} 
\label{f2}
\end{figure}
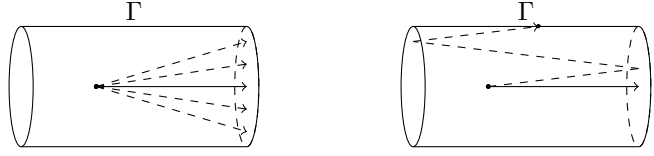

Heuristically the polynomial decay is due to the failure of the high frequency waves to concentrate exponentially along the geodesics that do not see the damping. In our geometry, the only geodesics that do not see the damping are those of momentum purely in $y$. Those geodesics are neutrally stable and hence delocalise spatially on $\Omega_x$: this explains the failure of such exponential concentration, similar to the case of interior damping. See Figure \ref{f2}. 

We remark the optimal $t^{-1/2}$-decay here is worse than that in the case of interior damping, for example, the $t^{-1+\delta}$-decay observed in \cite{aln14}. This is because those delocalised waves are damped more weakly by the boundary damping than the interior damping: the smaller the angle of incidence is when the geodesic hits the dissipative boundary, the weaker the dissipation of the waves concentrating near the geodesic is. When we perturb the neutrally trapped geodesic of momentum purely in $y$, a smaller perturbation leads to not only longer time for the perturbed geodesic to reach the damping, but also a smaller angle of incidence at the boundary, implying a weaker dissipation. This adverse phenomenon is absent in the case of interior damping and leads to the slower energy decay here. 

\subsubsection{Infinite Cylinders}We could also consider cylinders that are infinitely long along the cylindrical direction $y$. An infinite product cylinder has the boundary
\begin{equation}\label{1l19}
\Gamma=\pd\Omega=\pd\Omega_x\times\mathbb{R}_y.
\end{equation}

\begin{figure}
\begin{tikzpicture}[minimum size=0.01cm]
\draw[fill=gray!50] (0, 0.8) arc (90:270:-0.16 and 0.8) -- (6, -0.8) -- (6, -0.8) arc (270:90:-0.16 and 0.8) -- (6,0.8) -- cycle;
\draw (0, 0.8) arc (90:270:-0.16 and 0.8);
\draw [dashed] (0,0) ellipse (0.16 and 0.8);
\draw (0, -0.8) -- (6, -0.8);

\draw (6, 0.8) arc (90:270:-0.16 and 0.8);
\draw [dashed] (6,0) ellipse (0.16 and 0.8);

\draw (0,0.8) -- (6,0.8); 

\node at (3,1) {$\Gamma$}; 

\draw (3, 0.8) arc (90:270:-0.16 and 0.8);
\draw [dashed] (3,0) ellipse (0.16 and 0.8);
\fill [pattern=north east lines] (3,0) ellipse (0.16 and 0.8);
\node at (3.4,0) {$\Omega_x$};

\draw [dashed] (0, -0.8) -- (-0.5, -0.8);
\draw [dashed] (0, 0.8) -- (-0.5, 0.8);
\draw [dashed] (6, -0.8) -- (6.5, -0.8);
\draw [dashed] (6, 0.8) -- (6.5, 0.8);
\end{tikzpicture}
\caption{An example of an infinite product cylinder. It extends indefinitely to both ends.} 
\label{f3}
\end{figure}
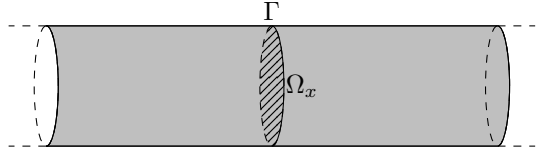
Consider the wave equation with damping from the boundary:
\begin{gather}\label{1l10}
\left(\pd_t^2+\Delta\right)u(t,p)=0, \ (t, p)\in \setr_t\times \Omega,\\
\pd_n u(t,p) +a(p)\pd_t\gamma u(t,p)+b(p)\gamma u(t,p)=0, (t, p)\in \setr_t\times \Gamma,\\
\label{1l11}
u(0, p)=u_0(p)\in H^2(\Omega), \ \pd_t u(0, p)=u_1(p)\in H^1(\Omega).
\end{gather}
We can show there is sharp polynomial decay under our damping assumption \ref{0ad}:
\begin{theorem}\label{0t4}
Let $u$ be the unique solution to the damped wave equation \eqref{1l10} on an infinite product cylinder $\Omega$ defined in Definition \ref{0gp}, impose the damping assumption \ref{0ad} and assume $a(p), b(p)\in C^{0,1/2+\delta}(\Gamma)$ for some $\delta>0$, and $b(p)\ge b_0>0$ on $\Gamma$. Then there is $C>0$ such that
\begin{equation}
E(u,t)^{\frac12}\le Ct^{-\frac12}\left(\|u_0\|_{H^2(\Omega)}+\|u_1\|_{H^1(\Omega)}\right),
\end{equation}
where $C$ is independent of the initial data $u_0$ and $u_1$.
\end{theorem}
\begin{remark}
\begin{enumerate}[wide]
\item This rate is also shown to be sharp in Proposition \ref{4t2}.
\item It is well known that waves on non-compact manifolds could have quasimodes concentrating at the zero frequency, and those quasimodes behave like solutions to the heat equation, which would require further analysis. See \cite{jr18}. We note that the further assumption that $b\ge b_0>0$ annihilates those quasimodes. 
\end{enumerate} 
\end{remark}

\subsubsection{Non-Product Cylinders}
We could consider more general geometric settings which are not product manifolds. 
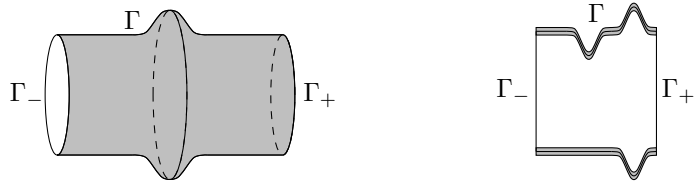
\begin{figure}
\begin{tikzpicture}[minimum size=0.01cm]
\draw[fill=gray!50] (0, 0.8) arc (90:270:-0.16 and 0.8) -- (1, -0.8) -- plot [smooth] coordinates {(1,-0.8) (1.1, -0.805) (1.2,-0.85) (1.4, -1.1) (1.6, -1.1) (1.8, -0.85) (1.9, -0.805) (2, -0.8)} -- (2, -0.8) -- (3, -0.8) arc (270:90:-0.16 and 0.8) -- (3,0.8) -- (2,0.8) -- plot [smooth] coordinates {(2, 0.8) (1.9, 0.805)(1.8, 0.85) (1.6, 1.1) (1.4, 1.1) (1.2,0.85) (1.1, 0.805) (1,0.8)} -- (1, 0.8) -- (0, 0.8) -- cycle;
\draw (0,0) ellipse (0.16 and 0.8);
\node at (-0.4,0) {$\Gamma_-$};
\draw (0, -0.8) -- (1, -0.8);
\draw (2, -0.8) -- (3, -0.8);

\draw (3, 0.8) arc (90:270:-0.16 and 0.8);
\draw [dashed] (3,0) ellipse (0.16 and 0.8);

\node at (3.5,0) {$\Gamma_+$};

\draw (0,0.8) -- (1,0.8); 
\draw (2,0.8) -- (3,0.8); 

\node at (1,1) {$\Gamma$}; 
\draw plot [smooth] coordinates {(1,0.8) (1.1, 0.805) (1.2,0.85) (1.4, 1.1) (1.6, 1.1) (1.8, 0.85) (1.9, 0.805) (2, 0.8)};
\draw plot [smooth] coordinates {(1,-0.8) (1.1, -0.805) (1.2,-0.85) (1.4, -1.1) (1.6, -1.1) (1.8, -0.85) (1.9, -0.805) (2, -0.8)};

\draw (1.5, 1.125) arc (90:270:-0.225 and 1.125);
\draw [dashed] (1.5, 0) ellipse (-0.225 and 1.125);
\end{tikzpicture}\hspace{5em}
\begin{tikzpicture}[minimum size=0.01cm]
%\draw[fill=gray!50] (0, 0.75) -- (1.6, 0.75) -- (1.6, 0.85) -- (0, 0.85) -- cycle;
%\draw[fill=gray!50] (0, -0.75) -- (1.6, -0.75) -- (1.6, -0.85) -- (0, -0.85) -- cycle;
\draw[fill=gray!50] (0, 0.85) -- (0.4, 0.85) -- plot [smooth] coordinates {(0.4,0.85) (0.5, 0.84) (0.55, 0.78) (0.66, 0.55) (0.74, 0.55) (0.85, 0.78) (0.9, 0.84) (1, 0.85)}  -- plot [smooth] coordinates {(1.0,0.85) (1.1, 0.86) (1.15, 0.92) (1.26, 1.15) (1.34, 1.15) (1.45, 0.92) (1.5, 0.86) (1.6, 0.85)} -- (1.6, 0.85) -- (1.6, 0.75) -- plot [smooth] coordinates {(1.6, 0.75) (1.5, 0.76) (1.45, 0.82) (1.34, 1.05) (1.26, 1.05) (1.15, 0.82) (1.1, 0.76) (1.0,0.75)} -- plot [smooth] coordinates {(1, 0.75) (0.9, 0.74) (0.85, 0.68) (0.74, 0.45) (0.66, 0.45) (0.55, 0.68) (0.5, 0.74) (0.4,0.75)} -- (0, 0.75) -- cycle;

\draw[fill=gray!50] (0, -0.75) -- (1, -0.75) -- plot [smooth] coordinates {(1.0,-0.75) (1.1, -0.76) (1.15, -0.82) (1.26, -1.05) (1.34, -1.05) (1.45, -0.82) (1.5, -0.76) (1.6, -0.75)} -- (1.6, -0.75) -- (1.6, -0.85) --  plot [smooth] coordinates {(1.6, -0.85) (1.5, -0.86) (1.45, -0.92) (1.34, -1.15) (1.26, -1.15) (1.15, -0.92) (1.1, -0.86) (1.0,-0.85)} -- (0, -0.85) -- cycle;
\draw (0, -0.8) -- (0, 0.8);
\draw (0, -0.8) -- (1, -0.8);
\draw (1.6, -0.8) -- (1.6, 0.8);
\draw (0.4, 0.8) -- (0, 0.8);
\node at (-0.3, 0) {$\Gamma_-$};
\node at (1.9, 0) {$\Gamma_+$};
\node at (0.8, 1.05) {$\Gamma$};

\draw plot [smooth] coordinates {(1.0,0.8) (1.1, 0.81) (1.15, 0.87) (1.26, 1.1) (1.34, 1.1) (1.45, 0.87) (1.5, 0.81) (1.6, 0.8)};
\draw plot [smooth] coordinates {(0.4,0.8) (0.5, 0.79) (0.55, 0.73) (0.66, 0.5) (0.74, 0.5) (0.85, 0.73) (0.9, 0.79) (1, 0.8)};
\draw plot [smooth] coordinates {(1.0,-0.8) (1.1, -0.81) (1.15, -0.87) (1.26, -1.1) (1.34, -1.1) (1.45, -0.87) (1.5, -0.81) (1.6, -0.8)};
\end{tikzpicture}
\caption{Two examples of compact non-product cylinders. }
\label{f4} 
\end{figure}
\begin{definition}[Compact Non-Product Cylinders] \label{0gcn}
A compact non-product cylinder $\Omega$ with star-shaped cross-sections is a bounded $C^{0,1}$-domain $\Omega\subset\mathbb{B}_x^{d}(0, c_1)\times[-1,1]_y$ for some radius $c_1>0$. We assume that $\pd\Omega\cap (-1,1)_y$ is $C^{1,1}$, and the extension of $\Omega$ by reflection past $y=\pm1$ has a boundary that is $C^{1,1}$ near $y=\pm 1$. Assume further that all its cross-sections $\Omega(y)=\{x: (x,y)\in\Omega\}$ parametrised by $y\in[-1,1]$ are non-empty and star-shaped, that is, there exists $c_0>0$ independent of $y$ with $x\cdot n_x\ge c_0>0$ almost everywhere on the boundary of the cross-sections $\pd\Omega(y)=\pd(\Omega(y))$, where $n_x$ is the outward-pointing normal defined almost everywhere. See Figure \ref{f4} for examples. 
\end{definition}
Note that a compact product cylinder is a special example of a compact non-product cylinder. We will then have the boundary decomposition $\pd\Omega=\Gamma\cup(\Gamma_-\sqcup\Gamma_+)$ where
\begin{equation}\label{1l2}
\Gamma=\overline{\{(x,y):x\in\pd\Omega(y), y\in(0,1)\}}, \ \Gamma_\pm=\Omega(\pm 1).
\end{equation}
% We could also consider dampings that also live on $\Gamma_0\sqcup\Gamma_1$ that turn on and off regularly:
% \begin{assumption}[Two-Sided Damping Assumption]\label{0a2}
% Let the damping function $a(x,y)$ in $L^\infty(\pd\Omega)$ be that
% \begin{gather}
% 0<a_0\le a(x,y)\le a_1, \ (x,y)\in\Gamma,\\
% 0\le \underline{a}(x,y)\le a(x,y)\le c_2\underline{a}(x,y), \ (x,y)\in \Gamma_0\sqcup\Gamma_1,
% \end{gather}
% for some $c_2\ge 1$, $\underline{a}(x,y)\in C^{0,1}(\Gamma_0\sqcup\Gamma_1)$ such that
% \begin{equation}
% \abs{\nabla_x \underline{a}(x,y)}\le C \sqrt{\underline{a}(x,y)}
% \end{equation}
% essentially on $\Gamma_0\sqcup\Gamma_1$.
% \end{assumption}
% It is noted that the one-sided damping Assumption \ref{0ad} is a special case of the two-sided one. Now consider the damped wave equations \eqref{1l1} with the two-sided damping Assumption \ref{0a2}. 
\begin{theorem}\label{0t6}
Let $u$ be the unique solution to the damped wave equation \eqref{1l1} on a compact non-product cylinder $\Omega$ defined in Definition \ref{0gcn}, impose the damping assumption \ref{0ad} and assume $a(p), b(p)\in C^{0,1/2+\delta}(\Gamma)$ for some $\delta>0$. Then there is $C>0$ such that
\begin{equation}
E(u,t)^{\frac12}\le Ct^{-\frac13}\left(\|u_0\|_{H^2(\Omega)}+\|u_1\|_{H^1(\Omega)}\right),
\end{equation}
where $C$ is independent of the initial data $u_0$ and $u_1$.
\end{theorem}
\begin{remark}\label{0t7}
This result also works with the Dirichlet or periodic boundary conditions on $\Gamma_\pm$. Note that using the periodic boundary condition on $\Gamma_\pm$ requires that $\Gamma_-$ is isometric to $\Gamma_+$, and that we use the periodic extension instead of reflection past $y=\pm1$ in Definition \ref{0gcn}. See Figure \ref{f6}. 
\end{remark}
We can also consider an infinite non-product cylinder. 

\begin{definition}[Infinite Non-Product Cylinders]\label{0gin}Parametrise $\mathbb{R}^d_x$ and $\mathbb{B}^d_x$ by $x=rw$ using spherical coordinates $(r,w)\in \mathbb{R}_r\times\mathbb{S}^{d-1}_w$. An infinite non-product cylinder $\Omega$ with star-shaped cross-sections is an unbounded $C^{1,1}$-domain $\Omega\subset\mathbb{B}_{r,w}^{d}(0, c_1)\times\mathbb{R}_y$ for some radius $c_1>0$, such that $\Omega=(\Phi^{-1})^*(\mathbb{B}_{r',w'}^d\times\mathbb{R}_{y'})$ via some diffeomorphism
\begin{equation}
\Phi(r',w',y')=(r'\rho(w', y'),w',y'),
\end{equation}
in which $\rho: \mathbb{S}^{d-1}_w\times \mathbb{R}_y\rightarrow [c_0', c_1]$ is a $C^{1,1}$-scaling function, for some $c_0'>0$. Note that under such assumption, all cross-sections $\Omega(y)=\{x: (x,y)\in\Omega\}$ parametrised by $y\in\mathbb{R}$ are necessarily non-empty and star-shaped, that is, there exists $c_0>0$ independent of $y$ with $x\cdot n_x\ge c_0>0$ almost everywhere on the boundary of the cross-sections $\pd\Omega(y)=\pd(\Omega(y))$, where $n_x$ is the outward-pointing normal defined almost everywhere. See Figure \ref{f5} for examples. 
\end{definition}

\begin{figure}
\begin{tikzpicture}[minimum size=0.01cm]
\draw[fill=gray!50] (-3, 0.8) -- plot [domain=-3:-0.5, smooth, variable=\x] ({\x}, {1/(1.15*sqrt{-\x+0.5})+0.399}) -- plot [smooth] coordinates {(-0.51, 0.7575+0.399) (0,0.925+0.399) (0.51, 0.7575+0.399)} -- plot [domain=0.5:3, smooth, variable=\x] ({\x}, {1/(1.15*sqrt{\x+0.5})+0.399}) -- (3,0.8) arc (90:270:-0.16 and 0.8) -- plot [domain=3:0.5, smooth, variable=\x] ({\x}, {-1/(1.15*sqrt{\x+0.5})-0.399}) -- plot [smooth] coordinates {(0.51, -0.7575-0.399) (0,-0.925-0.399) (-0.51, -0.7575-0.399)} -- plot [domain=-0.5:-3, smooth, variable=\x] ({\x}, {-1/(1.15*sqrt{-\x+0.5})-0.399}) -- (-3, -0.8) arc (270:90:-0.16 and 0.8) -- cycle;

\draw (0, +0.925+0.399) arc (90:270:-0.2648 and 1.324);
\draw [dashed] (0, 0) ellipse (-0.2648 and 1.324);
\draw [dashed] (3, 0) ellipse (-0.16 and 0.8);

\draw [dashed] (-3,0) ellipse (0.16 and 0.8);
\draw [dashed] plot [domain=-4:-3, smooth, variable=\x] ({\x}, {1/(1.15*sqrt{-\x+0.5})+0.399});
\draw [dashed] plot [domain=-4:-3, smooth, variable=\x] ({\x}, {-1/(1.15*sqrt{-\x+0.5})-0.399});
\draw [dashed] plot [domain=3:4, smooth, variable=\x] ({\x}, {1/(1.15*sqrt{\x+0.5})+0.399});
\draw [dashed] plot [domain=3:4, smooth, variable=\x] ({\x}, {-1/(1.15*sqrt{\x+0.5})-0.399});

\node at (-1.5, 1.1) {$\Gamma$};

%\draw[scale=1, domain=0.7:2.3, smooth, variable=\x] plot ({-\x*\x-0.5}, {1/(1.15*\x)});
%\draw (-0.5, 0.7625) -- (0.5, 0.7625);

\end{tikzpicture}\\
\begin{tikzpicture}[minimum size=0.01cm]
\draw[fill=gray!50] (0, 0.8) arc (90:270:-0.16 and 0.8) -- (0, -0.8) -- plot [smooth] coordinates {(0,-0.8) (0.1, -0.805) (0.2,-0.85) (0.4, -1.1) (0.6, -1.1) (0.8, -0.85) (0.9, -0.805) (1, -0.8)} -- (1,-0.8) -- plot [smooth] coordinates {(1,-0.8) (1.1, -0.805) (1.2,-0.85) (1.4, -1.1) (1.6, -1.1) (1.8, -0.85) (1.9, -0.805) (2, -0.8)} -- (2,-0.8) -- plot [smooth] coordinates {(2.5,-0.8) (2.6, -0.805) (2.7,-0.85) (2.9, -1.1) (3.1, -1.1) (3.3, -0.85) (3.4, -0.805) (3.5, -0.8)} -- (3.5,-0.8) -- plot [smooth] coordinates {(4.5,-0.8) (4.6, -0.805) (4.7,-0.85) (4.9, -1.1) (5.1, -1.1) (5.3, -0.85) (5.4, -0.805) (5.5, -0.8)} -- (6, -0.8) arc (270:90:-0.16 and 0.8) -- (6,0.8) -- plot [smooth] coordinates { (5.5, 0.8) (5.4, 0.805) (5.3, 0.85) (5.1, 1.1) (4.9, 1.1) (4.7,0.85) (4.6, 0.805) (4.5,0.8)} -- (3.5,0.8) -- plot [smooth] coordinates {(3.5, 0.8) (3.4, 0.805) (3.3, 0.85) (3.1, 1.1) (2.9, 1.1) (2.7,0.85) (2.6, 0.805) (2.5,0.8)} -- plot [smooth] coordinates {(2, 0.8) (1.9, 0.805) (1.8, 0.85) (1.6, 1.1) (1.4, 1.1) (1.2,0.85) (1.1, 0.805) (1,0.8)} -- plot [smooth] coordinates {(1, 0.8) (0.9, 0.805) (0.8, 0.85) (0.6, 1.1) (0.4, 1.1) (0.2,0.85) (0.1, 0.805) (0,0.8)} --  (0, 0.8) -- cycle;

%\draw[fill=gray!50] (0, 0.8) arc (90:270:-0.3 and 0.8) -- (2.5, -0.8) -- plot [smooth] coordinates {(2.5,-0.8) (2.6, -0.805) (2.7,-0.85) (2.9, -1.1) (3.1, -1.1) (3.3, -0.85) (3.4, -0.805) (3.5, -0.8)} -- (3.5, -0.8) -- (6, -0.8) arc (270:90:-0.3 and 0.8) -- (6,0.8) -- (3.5,0.8) -- plot [smooth] coordinates {(3.5, 0.8) (3.4, 0.805) (3.3, 0.85) (3.1, 1.1) (2.9, 1.1) (2.7,0.85) (2.6, 0.805) (2.5,0.8)} -- (2.5, 0.8) -- (0, 0.8) -- cycle;

\draw (0, 0.8) arc (90:270:-0.16 and 0.8);
\draw [dashed] (0,0) ellipse (0.16 and 0.8);

%\draw (3.5, 0.8) arc (90:270:-0.3 and 0.8);
%\draw (2.5, 0.8) arc (90:270:-0.3 and 0.8);

\draw (6, 0.8) arc (90:270:-0.16 and 0.8);
\draw [dashed] (6,0) ellipse (0.16 and 0.8);

\draw [dashed] (0, -0.8) -- (-0.5, -0.8);
\draw [dashed] (0, 0.8) -- (-0.5, 0.8);
\draw [dashed] (6, -0.8) -- (6.5, -0.8);
\draw [dashed] (6, 0.8) -- (6.5, 0.8);

\draw (0.5, 1.125) arc (90:270:-0.225 and 1.125);
\draw [dashed] (0.5, 0) ellipse (-0.225 and 1.125);

\draw (1.5, 1.125) arc (90:270:-0.225 and 1.125);
\draw [dashed] (1.5, 0) ellipse (-0.225 and 1.125);

\draw (3, 1.125) arc (90:270:-0.225 and 1.125);
\draw [dashed] (3, 0) ellipse (-0.225 and 1.125);

\draw (5, 1.125) arc (90:270:-0.225 and 1.125);
\draw [dashed] (5, 0) ellipse (-0.225 and 1.125);

\node at (2, 1) {$\Gamma$};
\end{tikzpicture}\\
\caption{Two examples of infinite non-product cylinders. The left infinite end of the second example is straight, and the right infinite end consists of bounded perturbations that are further apart as $y\rightarrow+\infty$. } 
\label{f5}
\end{figure}
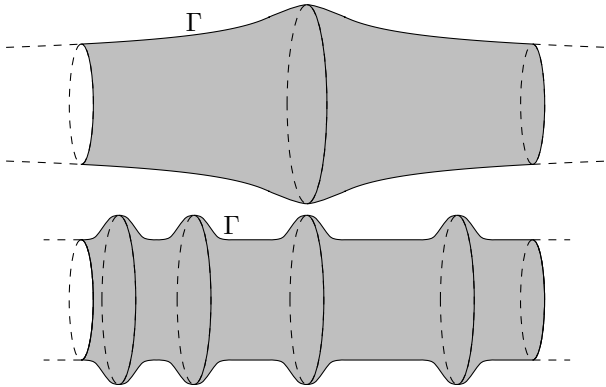

The $C^{1,1}$-assumption is essential to ensure the global elliptic regularity. Note that an infinite product cylinder is a special example of an infinite non-product cylinder. An infinite non-product cylinder has
\begin{equation}\label{1l3}
\pd\Omega=\Gamma=\overline{\{(x,y):x\in\pd\Omega(y), y\in\mathbb{R}\}}
\end{equation}
and a similar result holds if we assume $b(p)$ is bounded from below almost everywhere:
\begin{theorem}\label{0t5}
Let $u$ be the unique solution to the damped wave equation \eqref{1l10} on an infinite non-product cylinder $\Omega$ defined in Definition \ref{0gin}, impose the damping assumption \ref{0ad} and assume that $a(p), b(p)\in C^{0,1/2+\delta}(\Gamma)$ for some $\delta>0$, and $b\ge b_0>0$ on $\Gamma$. Then there is $C>0$ such that
\begin{equation}
E(u,t)^{\frac12}\le Ct^{-\frac13}\left(\|u_0\|_{H^2(\Omega)}+\|u_1\|_{H^1(\Omega)}\right),
\end{equation}
where $C$ is independent of the initial data $u_0$ and $u_1$.
\end{theorem}
Note that in those non-product geometries we were not able to obtain the sharp $t^{-1/2}$-decay as in Theorem \ref{0t1}, due to the fact that the second conormal derivatives at the boundary are not sufficiently damped. See Remark \ref{3t4}(2). It is not known to the author whether this $t^{-1/3}$-decay is optimal in those more generic geometric settings. 

The paper is organised in the following manner: in Section 2 we develop some resolvent estimates for the interior impedance problems on compact cylinders; in Section 3, we use those resolvent estimates to obtain the polynomial decay rates in Theorem \ref{0t1} and Theorem \ref{0t6}; in Section 4, we establish the resolvent estimates on infinite cylinders and prove Theorem \ref{0t4} and Theorem \ref{0t5}. 

\subsection{Related Works}
For exponential energy decay for damped wave equations under Geometric Control Conditions, the fundamental works are \cite{rt74} and \cite{blr92}. Also note \cite{lag83} for star-shaped domains, and \cite{bur97,cv02} on the cases of $C^3$-boundary and $C^1$-boundary. Logarithmic decay from the damping within arbitrarily small open sets has been well investigated in \cite{leb93,lr97,bur98}. Note due to recent Schrödinger observability results \cite{dj18,jin20}, interior damping on any open set of a compact hyperbolic surface also gives exponential decay. 

Polynomial decay on partially cylindrical domains has been extensively covered in \cite{lr05,bh07,phu07,phu08,nis09,nis13}. Also note the related non-concentration results in \cite{bz04,bz05,bhw07,bz12}. The relation between polynomial decay and regularity of interior damping on $\mathbb{T}^2$ has been well explained in \cite{aln14,kle19a,kle19b,dk20}. 

For damped wave equations on non-compact manifolds, note the exponential and logarithmic decay obtained on $\mathbb{R}^d$ in \cite{bj16}, and on asymptotically cylindrical and conic manifolds in \cite{wan20}. Global and local polynomial decay has been reported in \cite{wun17,jr18,roy18a}, and specifically on cylindrical domains in \cite{mr18,roy18b,roy15}.

Interior impedance problems have been well studied within the literature of numerical analysis. We refer the readers to \cite{bsw16,ms14,melenk95} and the references therein. 

% On manifolds with non-smooth boundaries, Burq in \cite{bur97} generalised this result to manifolds equipped with $C^2$-metrics with $C^3$-boundaries. On a manifold with merely $C^1$-boundary, only a $t^{-1}$-decay is known in \cite{cv02} to hold when the damping is everywhere on the boundary. 

\subsection{Acknowledgement}
The author is grateful to Jared Wunsch for numerous discussions around these results as well as many valuable comments on the manuscript.

\section{Resolvent Estimates for Interior Impedance Problems}
The energy decay rates for the waves damped from the boundary are closely related to the resolvent estimates for interior impedance problems: Consider on a compact non-product cylinder $\Omega$, 
\begin{gather}\label{2l25}
(\Delta-\lambda^2)u=f,\text{ on }\Omega,\\
i\pd_n u+a \lambda \gamma u+ib\gamma u=g,\text{ on }\Gamma,\\
\label{2l27}
\pd_n u=0,\text{ on }\Gamma_\pm,
\end{gather}
in which the damping functions $a(p), b(p)$ satisfy the damping assumption \ref{0ad}. We have the following resolvent estimates for the interior impedance problem:
\begin{proposition}\label{2t4}
Consider a solution $u\in H^2(\Omega)$ to the interior impedance problem \eqref{2l25} under the damping Assumption \ref{0ad}. We have the following estimates:
\begin{enumerate}[wide]
\item On a compact non-product cylinder $\Omega$ defined in Definition \ref{0gcn}, assume $f\in L^2(\Omega)$ and $g\in L^2(\Gamma)$. Then for any $\epsilon>0$, we have some $C>0$ independent of $\lambda, u, f, g$ such that for any $\abs{\lambda}\ge\epsilon$,
\begin{equation}\label{2l29}
\|u\|_{L^2(\Omega)}^2+(1+\lambda^2)^{-1}\|\nabla u\|_{L^2(\Omega)}^2\le C(1+\lambda^{2})\|f\|_{L^2(\Omega)}^2+C\|g\|_{L^2(\Gamma)}^2.
\end{equation}
The same estimate holds if we replace the homogeneous Neumann boundary condition \eqref{2l27} on $\Gamma_\pm$ by the homogeneous Dirichlet boundary condition
\begin{equation}\label{2l30}
\gamma u|_{\Gamma_\pm}=0.
\end{equation}
 
\item On a compact product cylinder $\Omega$ defined in Definition \ref{0gp}, assume further that $a,b,\pd_y a,\pd_y b\in C^{0,1/2+\delta}(\Gamma)$ for some $\delta>0$, $f\in H^1(\Omega)$, $g\in H^{3/2}(\Gamma)$. Then for any $\epsilon>0$, we have some $C>0$ independent of $\lambda, u, f, g$ such that for any $\abs{\lambda}\ge\epsilon$,
\begin{equation}\label{2l32}
\|u\|_{L^2(\Omega)}^2+(1+\lambda^2)^{-1}\|\nabla u\|_{L^2(\Omega)}^2\le C\|f\|_{H^1(\Omega)}^2+C\|g\|_{H^1(\Gamma)}^2.
\end{equation}

\item We could allow $\epsilon=0$ in (1) and (2), if one further assumes
\begin{equation}\label{2l33}
0< b_0\le b(p)\le b_1 
\end{equation}
essentially on $\Gamma$.

\end{enumerate}
\end{proposition}
\begin{remark}\label{2t7}
\begin{enumerate}[wide]
\item The estimate \eqref{2l29} is also sharp according to Proposition \ref{3t5}.
\item The periodic boundary condition on $\Omega_\pm$,
\begin{equation}\label{2l31}
\gamma u|_{\Gamma_-}=\gamma u|_{\Gamma_+}, \ \gamma \pd_y u|_{\Gamma_-}=\gamma \pd_y u|_{\Gamma_-}.
\end{equation}
also works in Proposition \ref{2t4}(1). See Remark \ref{0t7} for the domain requirement for the periodic boundary condition.
\item We could replace the homogeneous Neumann boundary condition on $\Gamma_\pm$ by the homogeneous Dirichlet condition or the periodic boundary condition and the same result in Proposition \ref{2t4}(2) will hold, if we assume further $f$ satisfies the same boundary condition. See Remark \ref{2t5}. 
\end{enumerate}
\end{remark}

Inner products and norms are defaulted to be those in $L^2(\Omega)$ if not otherwise specified. 
\begin{proposition}[High Frequency $L^2$-Resolvent Estimate]\label{2t1}
On a compact non-product cylinder $\Omega$ defined in Definition \ref{0gcn}, impose the damping assumption \ref{0ad}. Consider any solution $u(x;h)\in H^2(\Omega)$ to the semiclassical problem
\begin{gather}\label{2l1}
P_h u(p;h)=(h^2\Delta-1)u(p;h)=f(p;h), \ p\in\Omega,\\
\label{2l7}
\left(ih\pd_n \pm a(p)+ihb(p)\right)u(p;h)=g(p;h), \ p\in \Gamma,\\
\label{2l20}
ih\pd_n u(p;h)=0, \ p\in \Gamma_\pm,
\end{gather}
where $f\in L^2(\Omega)$ and $g\in L^2(\Gamma)$ at each $h$. Then for any $h_0>0$ there is some $C>0$ such that uniformly in $h\in(0, h_0]$ we have
\begin{equation}\label{2l12}
\left\|u\right\|_{H_h^1(\Omega)}^2\le C h^{-6}\left\|f\right\|^2_{L^2(\Omega)}+Ch^{-2}\|g\|^2_{L^2(\Gamma)}
\end{equation}
with 
\begin{gather}\label{2l10}
\left\|u\right\|_{L^2(\Gamma)}^2\le C h^{-4}\left\|f\right\|^2_{L^2(\Omega)}+C\|g\|^2_{L^2(\Gamma)},\\
\int_{\Gamma}\abs{h\nabla u}^2\le C h^{-4}\left\|f\right\|^2_{L^2(\Omega)}+C\|g\|^2_{L^2(\Gamma)}.
\end{gather}
And if we replace the boundary condition \eqref{2l20} on $\Gamma_\pm$ by the homogeneous Dirichlet boundary condition \eqref{2l30} or the periodic boundary condition \eqref{2l31}, the same estimates hold.
\end{proposition}
We need a Poincaré inequality to prove this result:
\begin{lemma}\label{2t2}
Let $\Omega$ be a compact non-product cylinder. For any $u\in H^1(\Omega)$ we have some $C>0$ independent of $u$ such that
\begin{equation}
\|u\|_{L^2(\Omega)}^2\le C\|\nabla_x u\|_{L^2(\Omega)}^2+C\int_\Gamma (x\cdot n_x)\abs{u}^2,
\end{equation}
where $n_x=\Pi_x n$ is the projection of the outward-pointing normal $n$ onto $\mathbb{R}^d_x$. 
\end{lemma}
\begin{proof}
Note
\begin{equation}
\int_\Gamma (x\cdot n_x)\abs{u}^2=\int_\Omega -\nabla^* \left(\abs{u}^2(x,0)\right)=d\|u\|^2+2\langle x\cdot\nabla_x u, u\rangle,
\end{equation}
where $\nabla^*$ is the divergence operator on $T\Omega$, locally reads $\sum_i X^i\pd_i\mapsto -\sum_i \pd_i X^i$. So 
\begin{equation}
\|u\|^2\le d^{-1}\int_\Gamma (x\cdot n_x)\abs{u}^2+8d^{-2}\|x\cdot\nabla_x u\|+\frac12\|u\|^2,
\end{equation}
the absorption of the last term with boundedness of $x$ concludes the proof. 
\end{proof}
\begin{remark}
Note that the boundedness in $y$ is not needed in this proof: only $\abs{x}$ needs to be uniformly bounded in $\Omega$. Therefore this lemma continues to work in the setting of cylinders infinite in $y$, that is, infinite non-product cylinders considered in Section 4. 
\end{remark}
\begin{proof}[Proof of Proposition \ref{2t1}]
1. We begin the proof by establishing some control over the $L^2$-size of $u$ and $\nabla u$ over $\Gamma$. Pair \eqref{2l1} with $u$ to see
\begin{multline}\label{2l2}
\langle P_hu, u\rangle=\|h\nabla u\|^2-\|u\|^2-h^2\langle \pd_n u, u\rangle_{\pd\Omega}=\|h\nabla u\|^2-\|u\|^2\\
\mp ih \langle au, u\rangle_{\Gamma}^2+h^2\langle bu, u\rangle_{\pd\Omega}\pm ih\langle g, u\rangle_{\Gamma}.
\end{multline}
Take the imaginary part of \eqref{2l2} to see
\begin{multline}\label{2l16}
\|u\|_{\Gamma}^2\le a_0^{-1}\abs{-h^{-1}\cim\langle P_h u, u\rangle}+a_0^{-1}\abs{\cre\langle g, u\rangle_{\Gamma}}\\
\le C\epsilon^{-1}h^{-4}\|f\|^2+\epsilon h^2\|u\|^2+C\epsilon^{-1}\|g\|_{\Gamma}^2+\epsilon\|u\|_{\Gamma}^2,
\end{multline}
which turns into 
\begin{equation}\label{2l6}
\|u\|_{\Gamma}^2\le C \epsilon^{-1}h^{-4}\|f\|^2+\epsilon h^2\|u\|^2+C\|g\|_{\Gamma}^2
\end{equation}
after the absorption of $\epsilon$-size boundary term. Take the real part of\eqref{2l2} to see
\begin{multline}\label{2l4}
\abs{\|h\nabla u\|^2-\|u\|^2}\le\abs{\cre\langle P_hu, u\rangle}+\abs{h\cim\langle g, u\rangle_{\Gamma}}+h^2b_1\|u\|_{\Gamma}^2\\
\le C\epsilon^{-1}h^{-2}\|f\|^2+\epsilon h^2\|u\|^2+Ch\|g\|_{\Gamma}^2+Ch\|u\|_{\Gamma}^2\\
\le \epsilon h^2\|u\|^2+C\epsilon^{-1}h^{-2}\|f\|^2+Ch\|g\|_{\Gamma}^2,
\end{multline}
where $C$ depends on $h_0$ and we used the fact that 
\begin{equation}
\|u\|_{\Gamma}^2\le C \epsilon^{-1}a_0^{-2}h^{-3}\|f\|^2+\epsilon h\|u\|^2+Ca_0^{-2}\|g\|_{\Gamma}^2
\end{equation}
via an easy modification of \eqref{2l16}. 

2. Now we claim we have an inequality related to the Rellich-type identity:
\begin{multline}\label{2l5}
\|u\|^2+\frac{1}{4a_0}\int_\Gamma a\abs{h\nabla u}^2\le C\|\sqrt{a}u\|^2_{\Gamma}+C\epsilon^{-1}h^{-2}\|f\|^2+C\|g\|_{\Gamma}^2\\
+\epsilon\|h\nabla_x u\|^2+\|h\pd_y u\|^2+\epsilon h^2\|u\|^2.
\end{multline}
We pair $P_h u$ with $V=-ihx.\nabla_x$ to obtain 
\begin{multline}\label{2l3}
\langle P_h u, -ihx\cdot\nabla_x u\rangle=ih^3\langle \nabla u, \nabla(x\cdot\nabla_x u)\rangle-ih \langle u, x\cdot\nabla_x u\rangle\\
-ih^3\langle \pd_n u, x\cdot\nabla_x u\rangle_{\Gamma}.
\end{multline}
We firstly estimate the first term on the right. Note 
\begin{equation}
-\nabla^*\left(\abs{\nabla u}^2(x, 0)\right)=d\abs{\nabla u}^2+2\cre\left(x^t\cdot\nabla_x\nabla u\cdot\nabla \overline{u} \right)
\end{equation}
and by the divergence theorem we have
\begin{equation}\label{2l40}
\int_\Gamma\abs{\nabla u}^2(x\cdot n_x)=\int_\Omega -\nabla^*\left(\abs{\nabla u}^2(x,0)\right)=d\|\nabla u\|^2+2\cre\langle x^t\cdot\nabla_x\nabla u,\nabla u\rangle
\end{equation}
as $n_x\equiv 0$ on $\Gamma_\pm$. With \eqref{2l40}, the imaginary part of the first term of \eqref{2l3} reads
\begin{multline}
\cim\left(ih^3\langle \nabla u, \nabla(x\cdot\nabla_x u)\rangle\right)=h^3 \|\nabla_x u\|^2+h^3\cre \langle x\cdot\nabla_x\nabla u, \nabla u\rangle\\
=h^3 \|\nabla u\|^2-h^3\|\pd_y u\|^2-\frac{dh^3}{2}\|\nabla u\|^2+\frac{h^3}{2}\int_\Gamma\abs{\nabla u}^2(x\cdot n_x)\\
=-\frac{h^3(d-2)}{2}\|\nabla u\|^2+\frac{h^3}{2}\int_\Gamma\abs{\nabla u}^2(x\cdot n_x)-h^3\|\pd_y u\|^2.
\end{multline}
By noting that the second term on the right hand side of \eqref{2l3} can be simplified via
\begin{equation}
-\langle u, x\cdot\nabla_x u\rangle=\langle x\cdot\nabla_x u,u\rangle+d\|u\|^2-\langle u, (x\cdot n_x)u\rangle_\Gamma
\end{equation}
and we can estimate the imaginary part of \eqref{2l3} by
\begin{multline}
h^{-1}\cim\langle P_hu, -ihx\cdot\nabla_x u\rangle=\frac{d}{2}\|u\|^2+\frac{1}2\int_\Gamma\abs{h\nabla u}^2(x\cdot n_x)-\frac{d-2}2\|h\nabla u\|^2\\
-\frac{1}2\langle u, (x\cdot n_x)u\rangle_\Gamma+\cim \langle \pm au+ihbu, x\cdot h\nabla_x u\rangle_{\Gamma}\mp\cim \langle g, x\cdot h\nabla_x u\rangle_{\Gamma}-\|h\pd_y u\|^2
\end{multline}
and hence
\begin{multline}
\frac{d}{2}\|u\|^2+\frac{1}2\int_\Gamma\abs{h\nabla u}^2(x\cdot n_x)\le\frac{d-2}2\|h\nabla u\|^2
+\frac{1}2\|\sqrt{x\cdot n_x}u\|^2_\Gamma+\epsilon\|h\nabla_x u\|^2\\
+C(1+h^2)\|u\|^2_{\Gamma}+\epsilon\int_{\Gamma} a\abs{h\nabla_x u}^2+C\epsilon^{-1}h^{-2}\|f\|^2+C\epsilon^{-1}\|g\|_{\Gamma}^2+\|h\pd_y u\|^2.
\end{multline}
As $x\cdot n_x\ge c_0$ on $\Gamma$, for small $\epsilon$,
\begin{multline}
\frac{d}{2}\|u\|^2+\frac{c_0}{4}\int_\Gamma \abs{h\nabla u}^2\le\frac{d-2}2\|h\nabla u\|^2
+C\|u\|^2_{\Gamma}\\
+C\epsilon^{-1}h^{-2}\|f\|^2+C\epsilon^{-1}\|g\|_{\Gamma}^2+\|h\pd_y u\|^2+\epsilon\|h\nabla_x u\|^2.
\end{multline}
Bring \eqref{2l4} in and observe
\begin{multline}\label{2l13}
\frac{d}{2}\|u\|^2+\frac{c_0}{4}\int_\Gamma \abs{h\nabla u}^2\le\frac{d-2}2\|u\|^2
+C\|u\|^2_{\Gamma}\\
+C\epsilon^{-1}h^{-2}\|f\|^2+C\epsilon^{-1}\|g\|_{\Gamma}^2+\|h\pd_y u\|^2+\epsilon\|h\nabla_x u\|^2+\epsilon h^2\|u\|^2.
\end{multline}
Absorb the first term on the right to obtain \eqref{2l5}. 

3. Use \eqref{2l5} together with \eqref{2l6} to estimate turn \eqref{2l4} into
\begin{equation}
\|h\nabla u\|^2\le C\epsilon^{-1}h^{-4}\|f\|^2+C\|g\|_{\Gamma}^2
+\epsilon\|h\nabla_x u\|^2+\|h\pd_y u\|^2+\epsilon h^2\|u\|^2.
\end{equation}
Subtract $\|h\pd_y u\|^2$ from both sides to see
\begin{equation}\label{2l11}
(1-\epsilon)\|h\nabla_x u\|^2\le C\epsilon^{-1}h^{-4}\|f\|^2+C\|g\|_{\Gamma}^2+\epsilon h^2\|u\|^2.
\end{equation}
From Lemma \ref{2t2} we have
\begin{equation}
\|u\|^2\le C\epsilon^{-1}h^{-6}\|f\|^2+Ch^{-2}\|g\|_{\Gamma}^2+\epsilon \|u\|^2.
\end{equation}
Absorb the $\epsilon$ term to see the estimate for $\|u\|$. Revisit \eqref{2l4}, \eqref{2l6} and \eqref{2l11} to conclude the proof of the case of the Neumann boundary condition \eqref{2l20} on $\Gamma_\pm$. 

4. For the cases with the homogeneous Dirichlet boundary condition \eqref{2l30} or the homogeneous periodic boundary condition \eqref{2l31}, it suffices to note
\begin{equation}\label{2l24}
\langle \pd_n u, u\rangle_{\Gamma_-\sqcup\Gamma_+}=0, \ \langle \pd_n u, x\cdot \nabla_x u\rangle_{\Gamma_-\sqcup\Gamma_+}=0.
\end{equation}
Then \eqref{2l2} and \eqref{2l3} still hold. The rest is the same. 
\end{proof}
We now have proved the high frequency part where $\abs{\lambda}\ge h_0^{-1}>0$ of Proposition \ref{2t4}(1). In order to obtain its high frequency counterpart in Proposition \ref{2t4}(2), we need an elliptic regularity result. So we put aside the proof of Proposition \ref{2t4} and prove a regularity result first. 
\begin{lemma}[Multiplier Estimate]\label{2t5}
Let $a\in C^{0,\frac12+\delta}(\Gamma)$ for $\delta>0$ on $\Gamma$ compact. Then $a(p)$ is a Sobolev multiplier on $H^{1/2}(\Gamma)$, that is, for each $u\in H^{1/2}(\Gamma)$, we have $au\in H^{1/2}(\Gamma)$. Moreover, there exists $C>0$ independent of $a$ and $u$ such that
\begin{equation}
\|au\|^2_{H^{1/2}}\le 2\|a\|_{L^\infty}^2\|u\|_{H^{1/2}}^2+C\|a\|_{C^{0,1/2+\delta}}^2\|u\|_{L^2}^2.
\end{equation}
In other words, we have $C^{0,\frac12+\delta}(\Gamma)\subset MH^{1/2}(\Gamma)$ in the sense of \cite{ms09}. 
\end{lemma}
\begin{proof}
One could follow \cite{ms09} to use the Strichartz functions to characterise $MH^{1/2}(\Gamma)$, but here we could show it directly following the expository notes of \cite{npv12}. Denote the Gagliardo seminorm on $H^{1/2}(\Gamma)$ by
\begin{equation}
\|u\|_{\dot{H}^{1/2}}^2=\int_\Gamma\int_\Gamma \frac{\abs{u(p)-u(q)}^2}{d_{\Gamma}(p,q)^{d+1}}dp dq,
\end{equation}
where $d_{\Gamma}$ is the distance function on $\Gamma$ and $dp$, $dq$ are the volume form on $\Gamma$. Note that $\|\cdot\|_{H^{1/2}}$ and $\|\cdot\|_{L^2}+\|\cdot\|_{\dot{H}^{1/2}}$ are equivalent norms on $H^{1/2}(\Gamma)$. Note
\begin{equation}\label{2l35}
\|au\|^2_{\dot{H}^{1/2}}\le 2\int_\Gamma\int_\Gamma \frac{\abs{a(p)u(p)-a(p)u(q)}^2}{d_{\Gamma}(p,q)^{d+1}}+\frac{\abs{a(p)u(q)-a(q)u(q)}^2}{d_{\Gamma}(p,q)^{d+1}} dp dq.
\end{equation}
The first term of \eqref{2l35} is bounded by
\begin{equation}
2\int_\Gamma\int_\Gamma \frac{\abs{a(p)}^2\abs{u(p)-u(q)}^2}{d_{\Gamma}(p,q)^{d+1}} dp dq\le 2\|a\|_{L^\infty}^2\|u\|^2_{\dot{H}^{1/2}}.
\end{equation}
Since $a\in C^{0,1/2+\delta}(\Gamma)$, the half of the second term could be decomposed as
\begin{multline}
\int_\Gamma\int_\Gamma \frac{\abs{u(q)}^2\abs{a(p)-a(q)}^2}{d_{\Gamma}(p,q)^{d+1}} dp dq
\le \|a\|_{C^{0,1/2+\delta}}^2\int_\Gamma\int_{\Gamma\cap B_{\delta}(y)} \frac{\abs{u(q)}^2dp dq}{d_{\Gamma}(p,q)^{d-2\delta}} \\+4\|a\|_{L^{\infty}}^2\int_\Gamma\int_{\Gamma\cap B^c_{\delta}(q)} \frac{\abs{u(q)}^2dp dq}{d_{\Gamma}(p,q)^{d+1}}\le C\|a\|_{C^{0,1/2+\delta}}^2\|u\|_{L^2}^2.
\end{multline}
Note $\|au\|_{L^2}\le \|a\|_{L^\infty}\|u\|_{L^2}$ to conclude the proof. 
\end{proof}
\begin{remark}
This lemma continues to work when $\Gamma$ is no longer compact, if we assume $a\in C^{0,1/2+\delta}\cap L^\infty$.
\end{remark}

% We follow Jerison-Kenig (??81) and (??93, theorem B)
% \begin{theorem}[Fabes-Jerison-Kenig]For $d\ge 1$, let $\Omega\subset\mathbb{R}^{d+1}$ be a bounded Lipschitz domain. Then for any $u\in H^1(\Omega)$ such that
% \begin{gather}
% \Delta u=f\in L^2(\Omega),\\
% \pd_n u=g\in L^2(\pd\Omega)
% \end{gather}
% we have indeed $u\in H^{3/2}(\Omega)$.
% \end{theorem}
% \begin{proof}
% Extend $f$ by $0$ to $\tilde f\in L^2(\mathbb{R}^{d+1})$. Convolve $\tilde f$ against the Newton potential $N$ of dimension $d+1$ to have
% \begin{equation}
% w(p)=\tilde f(p)*N(p)=\mathcal{F}^{-1}_{p\leftarrow\xi}[\xi^{-2}\mathcal{F}_{p\rightarrow\xi}\tilde f(\xi)]\in H^{2}(\Omega).
% \end{equation}
% Note $\Delta w=f$ in $\Omega$, $\nabla w\in H^1(\Omega)$ and hence $\pd_n w\in H^{1/2}(\pd\Omega)\subset L^2(\pd\Omega)$. Thus
% \begin{gather}
% \Delta (u-w)=0, \ (x, y)\in \Omega,\\
% \pd_n (u-w)=g-\pd_n w\in L^2(\pd\Omega).
% \end{gather}
% Invoke Jerison-Kenig (??81, theorem 3) to see $u-w\in H^{3/2}(\Omega)$ and hence so is $u$. 
% \end{proof}

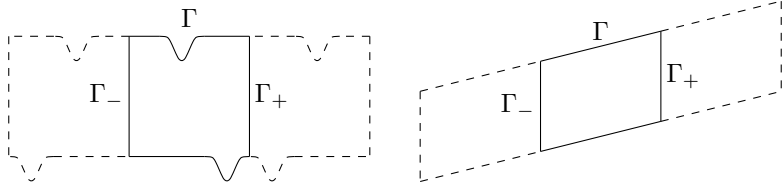
\begin{figure}
\begin{tikzpicture}[minimum size=0.01cm]
%\draw[fill=gray!50] (0, 0.75) -- (1.6, 0.75) -- (1.6, 0.85) -- (0, 0.85) -- cycle;
%\draw[fill=gray!50] (0, -0.75) -- (1.6, -0.75) -- (1.6, -0.85) -- (0, -0.85) -- cycle;

\draw (0, -0.8) -- (0, 0.8);
\draw (0, -0.8) -- (1, -0.8);
\draw (1.6, -0.8) -- (1.6, 0.8);
\draw (0.4, 0.8) -- (0, 0.8);
\draw (1,0.8) -- (1.6,0.8);
\node at (-0.3, 0) {$\Gamma_-$};
\node at (1.9, 0) {$\Gamma_+$};
\node at (0.8, 1.05) {$\Gamma$};

\draw plot [smooth] coordinates {(0.4,0.8) (0.5, 0.79) (0.55, 0.73) (0.66, 0.5) (0.74, 0.5) (0.85, 0.73) (0.9, 0.79) (1, 0.8)};
\draw plot [smooth] coordinates {(1.0,-0.8) (1.1, -0.81) (1.15, -0.87) (1.26, -1.1) (1.34, -1.1) (1.45, -0.87) (1.5, -0.81) (1.6, -0.8)};

\draw [dashed] (0, 0.8) -- (-0.4,0.8);
\draw [dashed] (-1, 0.8) -- (-1.6,0.8);
\draw [dashed] (-1.6, 0.8) -- (-1.6,-0.8);
\draw [dashed] (-1, -0.8) -- (0,-0.8);
\draw [dashed] plot [smooth] coordinates {(-0.4,0.8) (-0.5, 0.79) (-0.55, 0.73) (-0.66, 0.5) (-0.74, 0.5) (-0.85, 0.73) (-0.9, 0.79) (-1, 0.8)};
\draw [dashed] plot [smooth] coordinates {(-1.0,-0.8) (-1.1, -0.81) (-1.15, -0.87) (-1.26, -1.1) (-1.34, -1.1) (-1.45, -0.87) (-1.5, -0.81) (-1.6, -0.8)};

\draw [dashed] (3.2, 0.8) -- (2.8,0.8);
\draw [dashed] (2.2, 0.8) -- (1.6,0.8);
\draw [dashed] (3.2, 0.8) -- (3.2,-0.8);
\draw [dashed] (2.2, -0.8) -- (3.2,-0.8);
\draw [dashed] plot [smooth] coordinates {(2.8,0.8) (2.7, 0.79) (2.65, 0.73) (2.54, 0.5) (2.46, 0.5) (2.35, 0.73) (2.3, 0.79) (2.2, 0.8)};
\draw [dashed] plot [smooth] coordinates {(2.2,-0.8) (2.1, -0.81) (2.05, -0.87) (1.94, -1.1) (1.86, -1.1) (1.75, -0.87) (1.7, -0.81) (1.6, -0.8)};
\end{tikzpicture}
\hspace{1.2em}
\begin{tikzpicture}[minimum size=0.01cm]
%\draw[fill=gray!50] (0, 0.75) -- (1.6, 0.75) -- (1.6, 0.85) -- (0, 0.85) -- cycle;
%\draw[fill=gray!50] (0, -0.75) -- (1.6, -0.75) -- (1.6, -0.85) -- (0, -0.85) -- cycle;

\draw (0, -0.8) -- (0, 0.4);
\draw (1.6, -0.4) -- (1.6, 0.8);

\draw (0,0.4) -- (1.6, 0.8);
\draw (0,-0.8) -- (1.6, -0.4);

\draw [dashed] (1.6, 0.8) -- (3.2, 1.2);
\draw [dashed] (1.6,-0.4) -- (3.2, 0);

\draw [dashed] (-1.6, 0) -- (0, 0.4);
\draw [dashed] (-1.6, -1.2) -- (0, -0.8);

\draw [dashed] (-1.6, 0) -- (-1.6,-1.2);
\draw [dashed] (3.2, 1.2) -- (3.2,0);

\node at (-0.3, -0.2) {$\Gamma_-$};
\node at (1.9, 0.2) {$\Gamma_+$};
\node at (0.8, 0.8) {$\Gamma$};

\end{tikzpicture}
\caption{Reflection on the left, and periodic extension on the right.} 
\label{f6}
\end{figure}

\begin{proposition}[Classical Elliptic Regularity up to Corners]\label{2t6}
Let $\Omega$ be a compact non-product cylinder defined in Definition \ref{0gcn}. Let $f\in L^2(\Omega)$, $g\in H^{1/2}(\Gamma)$ and $a, b\in C^{0,1/2+\delta}(\Gamma)$ for some $\delta>0$. Impose the damping assumption \ref{0ad}. When $b\not\equiv 0$ identically, for any $\cre{z}\ge 0$ there exists a unique weak solution $u\in H^{1}(\Omega)$ to the equation
\begin{gather}\label{2l36}
(\Delta +z^2)u=f,\text{ on }\Omega,\\
\pd_n u+a z u+bu=g,\text{ on }\Gamma,\\
\label{2l34}
\pd_n u=0,\text{ on }\Gamma_\pm.
\end{gather}
When $b\equiv 0$, the same statement holds for $\{\cre{z}\ge 0, z\neq 0\}$; and at $z=0$, there exists a weak solution in $H^1$ if and only if 
\begin{equation}\label{2l39}
\langle f, 1\rangle+\langle g,1\rangle_{\Gamma}=0,
\end{equation}
and such solution is instead unique modulo constant functions. Moreover, any such solution $u\in H^1(\Omega)$ must be in $H^{2}(\Omega)$. The above are still true if we replace the homogeneous Neumann boundary condition \eqref{2l34} by the homogeneous Dirichlet boundary condition \eqref{2l30} or the periodic boundary condition \eqref{2l31}. 
\end{proposition}
\begin{proof}
1. We begin with the regularity for the case of the Neumann boundary condition on $\Gamma_\pm$: assume $u\in H^1$ solves \eqref{2l36}, we claim $u$ is indeed $H^2$. We will use the method of images to remove the singularities at the corners $\Gamma_{-}\cap\Gamma$ and $\Gamma_{+}\cap\Gamma$. Define a longer cylinder $\tilde\Omega\subset\mathbb{R}^d_x\times[-3,3]_y$ by reflecting $\Omega$ past $\Gamma_-$ and $\Gamma_+$, with new boundary $\pd\tilde\Omega=\tilde\Gamma_+\cup\tilde\Gamma_-\cup\tilde\Gamma$ where
\begin{equation}
\tilde\Gamma_\pm=\pd\tilde\Omega\cap\{y=\mp3\}, \ \tilde\Gamma=\pd\tilde\Omega\cap(-3,3)_y.
\end{equation}
See Figure \ref{f6} for illustration. Reflect accordingly $u,a,b,f,g$ to get $\tilde u, \tilde a, \tilde b, \tilde f, \tilde g$. Consider a smooth cutoff function $\chi(y)\equiv 1$ on a small neighbourhood of $[-1,1]$ and supported inside $[-2,2]$. Then $\chi \tilde u$ solves
\begin{gather}
\Delta\chi\tilde u=\chi \tilde f-z^2\tilde u-(\pd_y^2\chi)\tilde u-2(\pd_y\chi)\pd_y\tilde u\in L^2(\tilde\Omega),\\
\pd_n\chi\tilde u=(-\tilde a z-\tilde b)\gamma\chi\tilde u+(\pd_n\chi)\gamma\tilde u\in H^{1/2}(\pd\tilde\Omega).
\end{gather}
Note that $\supp\chi\tilde u\cap\pd\tilde\Omega$ is indeed $C^{1,1}$. From the standard elliptic regularity on $C^{1,1}$ domains, we see $\chi \tilde u$ is indeed in $H^2(\tilde\Omega)$ locally, hence globally. See, for example \cite[Theorem 4.18]{mcl00} for details on the elliptic regularity on $C^{1,1}$-domains. Note that $\chi\equiv 1$ on $\supp u$ to conclude that $u\in H^2(\Omega)$.

2. Uniqueness: if there is a solution $u\in H^1(\Omega)$, it is the unique one in $H^1$. Note that we know $u\in H^2(\Omega)$ indeed, and it then suffices to show that the only $H^2$-solution to the homogeneous problem
\begin{gather}\label{2l37}
(\Delta +z^2)u=0,\text{ on }\Omega,\\
\pd_n u+a z u+bu=0,\text{ on }\Gamma,\\
\pd_n u=0, \text{ on }\Gamma_\pm.
\end{gather}
is trivial. Pair \eqref{2l37} with $u$, use Green's first identity and take the real and imaginary parts to see
\begin{gather}\label{2l38}
\|\nabla u\|^2+((\cre z)^2-(\cim z)^2)\|u\|^2+\cre z\|\sqrt{a}u\|^2_{\pd\Omega}+\|\sqrt{b}u\|^2_{\pd\Omega}=0\\
\cim z(2\cre z\|u\|^2+\|\sqrt{a}u\|_{\pd\Omega}^2)=0 
\end{gather}
In the case of $\cre z>0$, it is easy to deduce that $u=0$ in $L^2$ and hence in $H^1$. When $\cre z=0$ and $\cim z=\lambda\neq 0$, invoke estimate \eqref{2l12} with $h=\abs{\lambda}^{-1}>0$ from Proposition \ref{2t1} to observe $u=0$. When $z=0$, \eqref{2l38} reduces to 
\begin{equation}
\|\nabla u\|^2+\|\sqrt{b}u\|^2_{\pd\Omega}=0
\end{equation}
and hence $\nabla u=0$ and $u=u_0$ a constant on $\Omega$. Then
\begin{equation}
\|u-u_0\|_{\pd\Omega}^2\le C\|u-u_0\|_{H^{1}}^2=C\|u-u_0\|^2+C\|\nabla u-\nabla u_0\|^2=0,
\end{equation}
whence $u=u_0$ almost everywhere on $\pd\Omega$. When $b$ is not identically $0$ on $\pd\Omega$, we have $u_0=0$. 

3. Existence: we claim there exists $u\in H^1(\Omega)$ solving \eqref{2l36}. This is \cite[Theorem 4.11]{mcl00} where we use the coerciveness of the bilinear form 
\begin{equation}
B(u, v)=\langle (\Delta+z^2)u, v\rangle+\langle \pd_n u, v\rangle_\Gamma =\langle \nabla u, \nabla v\rangle+z^2\langle u, v\rangle
\end{equation}
to establish the Fredholm alternative for the problem. The uniqueness of the cases except when $b\equiv 0$ and $z=0$ implies the existence. When $b\equiv 0$ and $z=0$, the only solutions to the homogeneous problem are the constants, and hence we have the existence if and only if \eqref{2l39} holds. 

4. For the homogeneous Dirichlet boundary condition on $\Gamma_\pm$, we make the following modification: in Step 3, we will use the same coercive form but instead on
\begin{equation}
H^1_{\Gamma_\pm}(\Omega)=\{u\in H^1(\Omega): \gamma u|_{\Gamma_\pm}=0\}. 
\end{equation}  
See \cite[Theorem 4.10]{mcl00} for further details. For the periodic boundary condition on $\Gamma_\pm$, in Step 1 use periodic extensions instead of reflections at $\Gamma_\pm$, and in Step 4 use the same coercive form on 
\begin{equation}
H^1_{per}(\Omega)=\{u\in H^1(\Omega): \gamma u|_{\Gamma_-}=\gamma u|_{\Gamma_+}\}. 
\end{equation}
See Figure \ref{f6} for illustration. The rest are the same. 
\end{proof}
With this regularity result, we can prove a stronger high-frequency estimate on compact product cylinders, when $f,g$ are of better regularity than that in Proposition \ref{2t1}:

\begin{proposition}[High Frequency $H^1$-Resolvent Estimate]\label{2t3}
On a compact product cylinder $\Omega$ defined in Definition \ref{0gp}, impose the damping assumption \ref{0ad} and assume further that $a,b,\pd_y a, \pd_y b\in C^{0,1/2+\delta}(\Gamma)$. Consider any solution $u(p;h)\in H^2(\Omega)$ to
\begin{gather}\label{2l8}
P_h u(p;h)=(h^2\Delta-1)u(p;h)=f(p;h), \ p\in\Omega,\\
\left(ih\pd_n \pm a(p)+ihb(p)\right)u(p;h)=g(p;h), \ p\in \Gamma,\\
\label{2l9}
ih\pd_n u(p;h)=0, \ x\in \Gamma_\pm,
\end{gather}
for $f\in H^1(\Omega)$ and $g\in H^{3/2}(\Gamma)$ at each $h$. Then for any $h_0>0$ there is some $C>0$ such that uniformly in $h\in(0, h_0]$ we have
\begin{equation}
\left\|u\right\|_{H_h^1(\Omega)}^2\le C h^{-4}\left\|f\right\|^2_{H^1(\Omega)}+C\|g\|^2_{H^1(\Gamma)}.
\end{equation}
\end{proposition}
\begin{proof}
Let $w=h\pd_y u\in H^1(\Omega)$ and observe that $w$ satisfies
\begin{gather}
(h^2\Delta-1)w=h\pd_y f, \text{ on }\Omega,\\
\left(ih\pd_n \pm a+ihb\right)w=\tilde g, \text{ on } \Gamma,\\
w=0, \ \text{ on }\Gamma_\pm,
\end{gather}
where
\begin{equation}
\tilde g=h\pd_y g\mp h(\pd_y a) u\mp ih^2(\pd_y b) u.
\end{equation}
Here we used the fact that $\pd_y$ is tangential to $\Gamma$ and hence commutes with the trace operator $\gamma$ on $\Gamma$ almost everywhere, due to the product structure of $\Omega$. Note
\begin{equation}
\|\tilde g\|_{\Gamma}^2\le h^2\|\pd_y g\|_{\Gamma}^2+Ch^2\|u\|_{\Gamma}^2\le Ch^2\|g\|^2_{H^1(\Gamma)}+Ch^{-2}\|f\|^2
\end{equation}
from \eqref{2l10} and 
\begin{multline}
\|\tilde g\|_{H^{1/2}(\Gamma)}^2\le h^2\|\pd_yg\|^2_{H^{1/2}(\Gamma)}+h^2\|(\pd_y a) u\|^2_{H^{1/2}(\Gamma)}+h^{4}\|(\pd_y b) u\|_{H^{1/2}(\Gamma)}^2\\
\le h^2\|\pd_yg\|^2_{H^{1/2}(\Gamma)}+Ch^2\|u\|^2_{H^{1/2}(\Gamma)}
\end{multline}
from the Lemma \ref{2t5} with $\pd_y a, \pd_y b\in C^{0,1/2+\delta}(\Gamma)$. Hence $\tilde g\in H^{1/2}(\Gamma)$ and we invoke Proposition \ref{2t6} to conclude that $w\in H^2(\Omega)$. Therefore we have enough regularity to apply Proposition \ref{2t1} to $w$ to see
\begin{multline}
\|h\pd_y u\|^2=\|w\|^2\le Ch^{-6}\|h\pd_y f\|^2+Ch^{-2}\left(Ch^2\|g\|^2_{H^1(\Gamma)}+Ch^{-2}\|f\|^2\right)\\
\le C h^{-4}\|f\|_{H^1}^2+C\|g\|_{H^1(\Gamma)}^2.
\end{multline}
From \eqref{2l11} and \eqref{2l12} we have
\begin{equation}
\|h\nabla u\|^2\le C h^{-4}\left\|f\right\|_{H^1}^2+C\|g\|^2_{H^1(\Gamma)}
\end{equation}
and hence
\begin{equation}
\|u\|^2\le C h^{-4}\left\|f\right\|_{H^1}^2+C\|g\|^2_{H^1(\Gamma)}
\end{equation}
from \eqref{2l4}. 
\end{proof}
\begin{remark}\label{2t8}
We could prove the same corollary when replacing \eqref{2l9} by the homogeneous Dirichlet boundary condition \eqref{2l30} or the periodic boundary condition \eqref{2l31} on $\Gamma_\pm$, if we impose the same boundary condition on $f$, an assumption which arises naturally in the construction of semigroups in Section 3: see Remark \ref{3t6}. To do so, observe on $\Gamma_\pm$ we have
\begin{equation}
\pd_y^2 u=\Delta_x u-u-f,
\end{equation}
and that $\nabla_x$ is tangential to $\Gamma_\pm$. Then the terms from \eqref{2l24} will read
\begin{equation}
\langle \pd_n w, w\rangle_{\Gamma_-\sqcup\Gamma_+}=h\langle \pd_y^2 u, \pd_n u\rangle_{\Gamma_-\sqcup\Gamma_+}=h\langle \Delta_x u-u-f, \pd_n u\rangle_{\Gamma_-\sqcup\Gamma_+}=0
\end{equation}
and
\begin{equation}
\langle \pd_n w, x\cdot \nabla_x w\rangle_{\Gamma_-\sqcup\Gamma_+}=h\langle \Delta_x u-u-f, x\cdot\nabla_x \pd_n  u\rangle_{\Gamma_-\sqcup\Gamma_+}=0
\end{equation}
due to the boundary conditions on $\Gamma_\pm$. Hence we have restored the Step 4 of the proof of Proposition \ref{2t1} for $w=h\pd_y u$, with those boundary conditions for $u$, and get the corollary. 
\end{remark}

\begin{proof}[Proof of Proposition \ref{2t4}]
1. In the setting of Proposition \ref{2t4}(1), assume \eqref{2l29} is not true. Then there exist $\abs{\lambda_n}\ge \epsilon$ and
\begin{equation}\label{2l15}
\|u_n\|_{L^2}^2+\langle \lambda_n\rangle^{-2}\|\nabla u_n\|_{L^2}^2\equiv 1, \ \|f_n\|=o_n(\langle \lambda_n\rangle^{-1}), \ \|g_n\|_{\Gamma}=o_n(1)
\end{equation}
such that
\begin{gather}\label{2l14}
(\Delta-\lambda_n^2) u_n=f_n, \text{ on }\Omega,\\
i\pd_n u_n+a \lambda_n +ib u_n=g_n, \text{ on }\Gamma.
\end{gather}

2. The high frequency regime: if $\{\lambda_n\}$ is not bounded, assume $\lambda_n\rightarrow \pm\infty$ via a subsequence. Semiclassicalise the equation by setting $h=\abs{\lambda_n}^{-1}\rightarrow 0$:
\begin{gather}
(h^2\Delta-1) u_h=h^2f_h, \text{ on }\Omega,\\
ih\pd_n u_h\pm a u_h+ihb u_h=h g_h, \text{ on }\Gamma.
\end{gather}
Note that $f_h=o_{L^2}(h), g_h=o_{L^2(\Gamma)}(1)$ and from Proposition \ref{2t1} we have
\begin{equation}
0<\epsilon^{-1}\langle \epsilon\rangle \le\left\|u_h\right\|_{H_h^1}^2\le C h^{-6}\left\|h^2f_h\right\|^2+Ch^{-2}\|hg_h\|_{\Gamma}^2=o(1),
\end{equation}
which gives the desired contradiction.

3. The low frequency regime: if $\{\lambda_n\}$ is bounded, via a convergent subsequence we assume $\lambda_n\rightarrow\lambda$, where $\pm \lambda\ge \epsilon$. Fix $h_0=2\abs{\lambda}^{-1}$ and $h=\abs{\lambda_n}^{-1}\in (h_0/4, h_0]$ and semiclassicalise in manner similar to that of Step 2. We have
\begin{multline}
0<\epsilon^{-1}\langle \epsilon\rangle \le\|u_n\|^2+h^2\|\nabla u_n\|^2\le C(h_0/4)^{-2}\|f_n\|^2+C\|g_n\|_{\Gamma}^2\\
=C(h_0/4)^{-2}o_n(h_0^2)+o_n(1)=o_n(1),
\end{multline}
which also gives the contradiction. Hereby we have obtained \eqref{2l29} for $\epsilon>0$.

4. Now we reinforce our assumption to include \eqref{2l33} and deal with the zero frequency regime where $\lambda\rightarrow 0$. Pair \eqref{2l14} with $u$ to see
\begin{equation}
\langle f_n,u_n \rangle=\|\nabla u_n\|^2-\lambda_n^2\|u_n\|^2-i\langle (a\lambda_n+ib)u_n, u_n\rangle_\Gamma+i\langle g_n, u_n\rangle_{\Gamma},
\end{equation}
the real part of which leads to 
\begin{equation}
\|\nabla u_n\|^2+\langle bu_n, u_n\rangle_{\Gamma}\le \lambda_n^2\|u_n\|^2+\|f_n\|\|u_n\|+C\epsilon^{-1}\|g_n\|_{\Gamma}+\epsilon\|u_n\|_{\Gamma}^2.
\end{equation}
Use the fact that $b\ge b_0>0$ to absorb the $\epsilon$-term and we get
\begin{equation}
\|\nabla u_n\|^2+\frac{b_0}2\|u_n\|^2_{\Gamma}\le \lambda_n^2\|u_n\|^2+\|f_n\|\|u_n\|+C\epsilon^{-1}\|g_n\|_{\Gamma}=o_n(1),
\end{equation}
as $\|u_n\|\le 1$ and $\lambda_n\rightarrow 0$. So both $\|\nabla u_n\|^2$ and $\|u_n\|_{\Gamma}^2$ are $o_n(1)$. Now invoke Lemma \ref{2t2} to see $\|u_n\|=o_n(1)$, and therefore
\begin{equation}
1 \equiv \|u_n\|_{L^2}^2+\langle \lambda\rangle^{-2}\|\nabla u_n\|_{L^2}^2=o_n(1),
\end{equation}
contradicting \eqref{2l15}. This concludes the proof of \eqref{2l29} for $\epsilon\ge 0$ in Proposition \ref{2t4}(3).

5. The proof of Proposition \ref{2t4}(2) is similar to that of Proposition \ref{2t4}(2). One needs to invoke Proposition \ref{2t3} instead of Proposition \ref{2t1}. 
\end{proof}

\section{Evolution on Compact Cylinders}
The polynomial decay of energy in Theorem \ref{0t1} is based on the resolvent estimates for the generators of dissipative semigroups. We here set up the evolution problem. Let $X$ be $H^1(\Omega)\oplus L^2(\Omega)$. If $b\equiv 0$ on $\Gamma$, we further restrict $X$ to the set of $(\tilde u,\tilde v)\in H^1(\Omega)\oplus L^2(\Omega)$ that
\begin{equation}\label{3l21}
\int_\Omega \tilde v+\int_{\Gamma}a(x) \tilde u(x, y)=0,
\end{equation}
a codimension-$1$ subspace of $H^1\oplus L^2$. Let the generator be
\begin{equation}\label{3l22}
A=\begin{pmatrix}
0 & \id\\
-\Delta & 0
\end{pmatrix}: D(A)\rightarrow X
\end{equation}
whose domain is 
\begin{equation}\label{3l7}
D(A)=\{(u,v)\in X\cap(H^2(\Omega)\oplus H^1(\Omega)), 
(\pd_n u+a v+b u)|_\Gamma=0, \ \pd_n u|_{\Gamma_\pm}=0 \}.
\end{equation}
Note $D(A)$ is dense in $X$. Define the energy seminorm $E$ on $X$ via 
\begin{equation}\label{3l23}
\|(u,v)\|_E^2=\|\nabla u\|_{L^2(\Omega)}^2+\|v\|_{L^2(\Omega)}^2+\|\sqrt{b}u\|_{L^2(\Gamma)}^2.
\end{equation}
We refer to \cite{lr97} and \cite{nis13} for further details on those semigroups on $C^{1,1}$-manifolds. 

% The technical problem here, is that our compact cylinder $\Omega$ of concern here has corners of codimension 2 at $\pd\Omega_x\times \{0,1\}$, no matter how smooth the boundary of our cross-section $\Omega_x$ is. On general Lipschitz domains, it known that the classical elliptic regularity up to the boundary does not hold. For example when the domain has a non-convex polygon corner, by assuming $\Delta u$ is $C^\infty$ and $u$ is $H^{3/2}$ we could still have $u\notin H^2$. This observation implies that $\id-A$ is not surjective from $D(A)$ to $X$, thus $A$ is not maximal and does not generate a contraction semigroup on $X$, as Lagnese acutely observed in \cite{lag83}. We claim that when the cross-section $\Omega_x$ is convex bounded Lipschitz domain, the elliptic regularity continues to hold on $\Omega$ and the generator $A$ is maximal. 

\begin{proposition}\label{3t1}
On a compact non-product cylinder $\Omega$ defined in Definition \ref{0gcn}, under the damping assumption \ref{0ad}, the energy seminorm $E$ is a norm on $X$ and is equivalent to the $H^1\oplus L^2$-norm. Furthermore when $a,b\in C^{0,1/2+\delta}(\Gamma)$, the generator $A$ is a maximally dissipative with respect to the energy norm $E$. Hence $A$ generates a strongly continuous semigroup $e^{tA}$ which is a contraction.
\end{proposition}
\begin{proof}
1. We begin with showing the seminorm $E$ is point-separating. Let 
\begin{equation}
\|(u,v)\|_E^2=\|\nabla u\|^2+\|v\|^2+\|\sqrt{b}u\|_{\Gamma}^2=0.
\end{equation}
We have $v\equiv 0$ and $u\equiv u_0$, a constant almost everywhere on $\Omega$. Then
\begin{equation}
\|u-u_0\|_{\pd\Omega}^2\le C\|u-u_0\|_{H^{1}}^2=C\|u-u_0\|^2+C\|\nabla u-\nabla u_0\|^2=0,
\end{equation}
whence $u\equiv u_0$ almost everywhere on $\pd\Omega$. When $b$ is not identically $0$ on $\Gamma$, we have $u_0=0$. When $b\equiv 0$ on $\pd\Omega$, \eqref{3l21} and the fact $a>a_0$ almost everywhere on $\Gamma$ implies $u_0=0$. So we have $(u,v)=0\in H^1\times L^2$. So $E$ is point separating, hence a norm. Note $\|(u,v)\|_E\le \|(u,v)\|_{H^1\times L^2}$, so the inclusion $X\hookrightarrow E$ is bijective and bounded, and is then an open map and two norms are equivalent. 

2. We claim $\id-A$ is bijective from $D(A)$ onto $X$. It suffices to show for any $(\tilde u, \tilde v)$ in $X$, we have
\begin{equation}\label{3l35}
(\id-A)\begin{pmatrix}
u\\
v
\end{pmatrix}=
\begin{pmatrix}
u-v\\
\Delta u+v
\end{pmatrix}=
\begin{pmatrix}
\tilde u\\
\tilde v
\end{pmatrix},
\end{equation}
for unique $(u,v)\in D(A)$. By taking $v=u-\tilde u$ we have a Robin boundary value problem for the Helmholtz equation:
\begin{gather}
\Delta u+u=\tilde u+\tilde v\in L^2(\Omega),\\
\label{3l36}
\pd_n u+a u+bu=a \tilde u\in L^2(\Gamma),\\
\pd_n u=0,\text{ on }\Gamma_\pm.
\end{gather}
When $b\not\equiv 0$, $X=H^1\oplus L^2$ and by Proposition \ref{2t6} we know there exists a unique $u\in H^2$ that solves the problem, and hence $(u, u-\tilde u)$ is the unique one in $D(A)$ to be sent by $A$ to $(\tilde u, \tilde v)$. When $b\equiv 0$, $X$ is a codimension-1 subspace of $H^1\oplus L^2$, and we have in addition to the regularity of $\tilde u, \tilde v$, that 
\begin{equation}
\langle \tilde v, 1\rangle+\langle a\tilde u, 1\rangle_{\Gamma}=0
\end{equation}
by the assumption \eqref{3l21}. We also need to verify this extra property for $u,v$ to claim $(u,v)\in D(A)$. Note in this case we have
\begin{equation}
\langle \tilde v, 1\rangle=\langle \Delta u+v, 1\rangle=\langle v,1\rangle-\langle \pd_n u, 1\rangle_{\pd\Omega}=\langle v,1\rangle+\langle a u, 1\rangle_{\Gamma}-\langle a \tilde u, 1\rangle_{\Gamma},
\end{equation}
from \eqref{3l35} and \eqref{3l36}. This implies
\begin{equation}
\langle v, 1\rangle+\langle a u, 1\rangle_{\Gamma}=\langle \tilde v, 1\rangle+\langle a\tilde u, 1\rangle_{\Gamma}=0,
\end{equation}
as claimed. The rest is the same as for the case when $b\not\equiv 0$. 

3. It is straightforward to compute
\begin{equation}\label{3l5}
\cre{\langle A(u,v), (u,v)\rangle_E}=-\int_{\Gamma} a(x, y)\abs{v}^2\le 0,
\end{equation} 
which demonstrates the dissipative nature of $A$. 
\end{proof}

\begin{remark}\label{3t6}
We give brief remarks on modifications of the semigroups for the Dirichlet boundary condition and the periodic boundary condition on $\Gamma_\pm$. Note that we always keep the impedance boundary condition on $\Gamma$. 
\begin{enumerate}[wide]
\item Consider the periodic boundary condition \eqref{1l17} on $\Gamma_\pm$. 
% For $s\ge0$ let $H^s_{per}(\Omega)$ denote the $H^s$-closure of $C_{per}^\infty(\Omega)$, the set of smooth functions $u$ on $\Omega$ such that $u|_{\Gamma_0}=u|_{\Gamma_1}$. Let $X$ be $H^1_{per}(\Omega)\times L^2(\Omega)$, and if $b\equiv 0$ on $\Gamma$ we further restrict $X$ to the set of $(\tilde u, \tilde v)\in H^1_{per}(\Omega)\times L^2(\Omega)$ that
Let
\begin{equation}\label{3l26}
X=\{(\tilde u,\tilde v)\in H^1(\Omega)\oplus L^2(\Omega): \gamma \tilde u|_{\Gamma_-}= \gamma \tilde u|_{\Gamma_+}\}
\end{equation}
and if $b\equiv 0$ on $\Gamma$ we further restrict $X$ from the above-mentioned set to
\begin{equation}
\int_\Omega \tilde v+\int_{\Gamma}a(x) \tilde u(x, y)=0.
\end{equation}
And let
\begin{multline}
D(A)=X\cap\{(u,v)\in H^2(\Omega)\oplus H^1(\Omega): \gamma \tilde u|_{\Gamma_-}= \gamma \tilde u|_{\Gamma_+}, \gamma \tilde v|_{\Gamma_-}= \gamma \tilde v|_{\Gamma_+}, \\
\gamma(\pd_n u+av+b u)|_{\Gamma}\equiv 0\}
\end{multline}
with the energy norm 
\begin{equation}\label{3l24}
\|(u,v)\|_E^2=\|\nabla u\|_{L^2(\Omega)}^2+\|v\|_{L^2(\Omega)}^2+\|\sqrt{b}u\|_{L^2(\Gamma)}^2.
\end{equation}
Proposition \ref{3t1} holds with infinitesimal changes. 

\item Consider Dirichlet boundary condition \eqref{1l18}. No matter whether $b$ vanishes identically or not, let
\begin{equation}\label{3l27}
X=\{(\tilde u,\tilde v)\in H^1(\Omega)\oplus L^2(\Omega): \gamma \tilde u|_{\Gamma_\pm}\equiv 0\}
\end{equation}
and 
\begin{multline}
D(A)=X\cap\{(u,v)\in H^2(\Omega)\oplus H^1(\Omega): \gamma u|_{\Gamma_-\sqcup\Gamma_+}\equiv 0, \gamma v|_{\Gamma_-\sqcup\Gamma_+}\equiv 0, \\
\gamma(\pd_n u+av+b u)|_{\Gamma}\equiv 0\}
\end{multline}
with the energy norm \eqref{3l24}. Proposition \ref{3t1} still holds from the observation that $u_0\equiv 0$ on $\Gamma_\pm$. 
\end{enumerate}
% For the periodic boundary condition \eqref{1l17}, we could replace \eqref{3l7} only on $\Gamma_0\sqcup\Gamma_1$ by 

% \eqref{2l31} and
% \begin{equation}
% v(x,0)=v(x,1), \ \pd_y v(x,0)=\pd_y v(x,1), \ x\in\Pi_x(\Gamma_0),
% \end{equation}
% while keeping \eqref{3l21} with $\pd\Omega$ replaced by $\Gamma$ when $b\equiv 0$ on $\Gamma$. For the Dirichlet boundary condition \eqref{1l18}, we could replace \eqref{3l7} only on $\Gamma_0\sqcup\Gamma_1$ by \eqref{2l30} and
% \begin{equation}
% v(x,y)=0, \ (x,y)\in \Gamma_0\sqcup\Gamma_1
% \end{equation}
% and remove \eqref{3l21} no matter whether $b$ vanishes or not on $\Gamma$. Proposition \ref{3t1} still holds in those cases, while showing $u_0=0$ in Step 1 would be infinitesimally different. 
\end{remark}

\begin{proposition}\label{3t2}
On a compact non-product cylinder $\Omega$ defined in Definition \ref{0gcn}, impose the damping assumption \ref{0ad} and assume $a,b\in C^{0,1/2+\delta}(\Gamma)$. Then the generator $A:D(A)\rightarrow X$ has a discrete spectrum, no element of which is purely imaginary. Furthermore, there exists $\epsilon>0$ that for $\abs{\lambda}\ge \epsilon$ we have
\begin{equation}\label{3l8}
\|(A+i\lambda)^{-1}\|_{E\rightarrow E}\le C\abs{\lambda}^3.
\end{equation}
If we further assume $\Omega$ is a compact product cylinder defined in Definition \ref{0gp}, we have
\begin{equation}\label{3l9}
\|(A+i\lambda)^{-1}\|_{E\rightarrow E}\le C\abs{\lambda}^2,
\end{equation}
which improves \eqref{3l8}. 
\end{proposition}

\begin{proof}
1. We begin by showing the spectrum is discrete. As $M$ is compact, by the Rellich theorem we have $D(A)\hookrightarrow X$ compactly, and $\id-A$ is surjective from $D(A)$ to $X$. So $(\id-A)^{-1}$ is compact and $\operatorname{Spec}(A)$ is discrete. 

2. We will show $0\notin \operatorname{Spec}(A)$. Since it has a discrete spectrum, the operator has $0$ in the spectrum only if $0$ is an eigenvalue. Indeed when $A(u,v)=0$ we have $v=0, \Delta u=0$ and
\begin{equation}
\|u\|_E^2=\|\nabla u\|^2+\|v\|^2+\|\sqrt{b}u\|^2_{\Gamma}=0.
\end{equation}
This implies $(u,v)=0$. So $0$ is not in the spectrum of $A$, whose closedness further implies the existence of $\epsilon>0$ that $i(-2\epsilon, 2\epsilon)\cap\operatorname{Spec}(A)$ is empty.

3. We now assume on a compact product cylinder, \eqref{3l9} is not true, and there are $\pm\lambda_n\ge \epsilon$ and $(u_n, v_n)\in D(A)$ with
\begin{equation}\label{3l2}
\|(u_n, v_n)\|_E^2=\|\nabla u_n\|^2+\|v_n\|^2+\|\sqrt{b}u_n\|^2_{\Gamma}\equiv 1
\end{equation}
such that 
\begin{equation}\label{3l17}
(A\pm i\lambda_n)(u_n, v_n)=o_E(\abs{\lambda_n}^{-2})=o_X(\abs{\lambda_n}^{-2}), 
\end{equation}
which is
\begin{gather}
\pm i\lambda_n u_n+v_n=o_{H^1}(\abs{\lambda_n}^{-2})\\
-\Delta u_n\pm i\lambda_n v_n=o_{L^2}(\abs{\lambda_n}^{-2}),
%b^\frac12 (\pm i\lambda_n u_n+v_n)=o_{L^2(\Gamma)}(\abs{\lambda_n}^{-2})
\end{gather}
implying
\begin{gather}\label{3l6}
v_n=\mp i\lambda_n u_n+o_{H^1}(\abs{\lambda_n}^{-2})\\
\label{3l13}
\Delta u_n-\lambda_n^2 u_n=o_{L^2}(\abs{\lambda_n}^{-2})+o_{H^1}(\abs{\lambda_n}^{-1}).
\end{gather}
Note $\|u_n\|\le C$ as $(u_n,v_n)$ is also bounded in $H^1\oplus L^2$. Now \eqref{3l5} implies
\begin{multline}\label{3l19}
\|\sqrt{a}v_n\|^2_{\Gamma}= -\cre\langle A(u,v),(u,v)\rangle_E\\
=-\cre\langle (A\pm i\lambda_n)(u,v),(u,v)\rangle_E=o(\abs{\lambda_n}^{-2})
\end{multline}
and from \eqref{3l6} we have
\begin{equation}\label{3l14}
\|u_n\|_{\Gamma}\le a_0^{-1}\abs{\lambda_n}^{-1}\|\sqrt{a}v_n\|_{\pd\Omega}+o(\abs{\lambda_n}^{-3})=o(\abs{\lambda_n}^{-2}).
\end{equation}
Pair \eqref{3l13} with $u_n$ and take the real part to see via \eqref{3l7} that
\begin{multline}\label{3l1}
\|\nabla u_n\|^2+\|\sqrt{b}u_n\|_{\Gamma}^2 -\abs{\lambda_n}^2\|u_n\|^2=-\cre{\langle a v_n, u_n\rangle}_{\pd\Omega}+o(\abs{\lambda_n}^{-1})\|u_n\|\\
=o(\abs{\lambda_n}^{-3})+o(\abs{\lambda_n}^{-1})=o(\abs{\lambda_n}^{-1}).
\end{multline}
And we have the system
\begin{gather}
\label{3l16}
\Delta u_n-\lambda_n^2 u_n=o_{L^2}(\abs{\lambda_n}^{-2})+o_{H^1}(\abs{\lambda_n}^{-1}),\\
\label{3l15}
i\pd_n u_n\pm a\lambda_n u_n+ibu_n=a o_{H^{1/2}(\Gamma)}(\abs{\lambda_n}^{-2}).
\end{gather}

4. We will replace the quasimodes damped by $a$ and $b$, by those damped by some smaller but more regular damping functions $\underline{a}$ and $\underline{b}$, constituting a monotonicity argument. Consider the replacement damping functions $\underline{a}(p),\underline{b}(p)\in C^{0,1/2+\delta}(\Gamma)\cap L^\infty(\Gamma)$ that
\begin{gather}
0< a_0\le \underline{a}(p)\le a(p), \ p\in \Gamma,\\
0\le \underline{b}(p)\le b(p), \ p\in \Gamma,\\
\pd_y\underline{a}(x,y),\pd_y\underline{b}(x, y)\in C^{0,1/2+\delta}(\Gamma)\cap L^\infty(\Gamma).
\end{gather}
The existence of the replacement damping functions can be easily shown by taking $\underline{a}=a_0$ on $\Gamma$, and $\underline{b}=0$. Note that the replacement damping functions are more regular than the original ones. Now consider from \eqref{3l15} we have on $\Gamma$ that
\begin{equation}
i\pd_n u_n\pm \underline{a}\lambda_n u_n+i\underline{b}u_n=\mp\lambda_n(a-\underline{a})u_n-i(b-\underline{b})u_n+a o_{H^{1/2}(\Gamma)}(\abs{\lambda_n}^{-2})=g_n.
\end{equation}
Keep in mind $\{g_n\}\subset H^{1/2}(\Gamma)$, though they are not uniformly $H^{1/2}(\Gamma)$-bounded. Note as $0\le a-\underline{a}\le a\le C\underline{a}$ and $0\le b-\underline{b}\le (b_1/a_0)a\le C\underline{a}$ for some $C\ge 1$, we have from \eqref{3l14} that
\begin{equation}
(a-\underline{a})u_n=\underline{a}\bigo_{\Gamma}(u_n)=\underline{a}o_{\Gamma}(\abs{\lambda_n}^{-2})
\end{equation}
and
\begin{equation}
(b-\underline{b})u_n=b_1/a_0\bigo_{\Gamma}(au_n)=\underline{a}\bigo_{\Gamma}(u_n)=\underline{a}o_{\Gamma}(\abs{\lambda_n}^{-2})
\end{equation}
due to \eqref{1l4}. Thus the equation \eqref{3l16} to \eqref{3l15} is reduced to
\begin{gather}
\label{3l11}
\Delta u_n-\lambda_n^2 u_n=o_{L^2}(\abs{\lambda_n}^{-2})+o_{H^1}(\abs{\lambda_n}^{-1}),\\
\label{3l12}
i\pd_n u_n\pm \underline{a}\lambda_n u_n+i\underline{b}u_n=g_n=o_{\Gamma}(\abs{\lambda_n}^{-1}),\ g_n\in H^{1/2}(\Gamma),
\end{gather}
a quasimode for $\underline{a}$ and $\underline{b}$. We have two regimes. 

5a. The low frequency regime: when $\{\lambda_n\}$ is bounded, we have $\epsilon\le \abs{\lambda_n}\le C$, where \eqref{3l11} and \eqref{3l12} read
\begin{gather}
\Delta u_n-\lambda_n^2 u_n=o_{L^2}(1),\\
i\pd_n u_n\pm \underline{a}\lambda_n u_n+i\underline{b}u_n=o_{\Gamma}(1).
\end{gather}
Apply the resolvent estimate \eqref{2l29} and we obtain 
\begin{equation}
\|u_n\|=o_n(1), \ \|v_n\|=o_n(1), \ \|\nabla u_n\|+\|b^{\frac12}u_n\|_\Gamma=o_n(1),
\end{equation}
the last of which was due to \eqref{3l1}. This contradicts that $\|(u_n, v_n)\|_E\equiv 1$. 

5b. The high frequency regime: when $\{\lambda_n\}$ is not bounded, via a subsequence $h=\abs{\lambda_n}^{-1}\rightarrow 0$. The equation then becomes
\begin{gather}
\label{3l3}
\Delta u_h-h^{-2} u_h=o_{L^2}(h^2)+o_{H^1}(h),\\
\label{3l10}
\pd_n u_h\pm \underline{a}h^{-1} u_h+i\underline{b}u_h=g_h=o_{\Gamma}(h),
\end{gather}
where $g_h\in H^{1/2}(\Gamma)$ at each $h$. Note from \eqref{3l1} we have $\|\nabla u_h\|+\|\sqrt{b}u_h\|_{\pd\Omega}\sim\|u_h\|/h$ and $\|v_h\|\sim\|u_h\|/h$ and 
\begin{equation}
\|u_h\|_{\Gamma}\le \|u_h\|_{\pd\Omega}\le C(\|u_h\|^2+\|u_h\|\|\nabla u_h\|)^{\frac12}\le Ch^{-\frac12} \|u_h\|.
\end{equation}
We have from \eqref{3l2} and $b\le b_1$ that
\begin{equation}
\|u_h\|\sim h/\sqrt{2}, \ \|\nabla u_h\|\sim 1/\sqrt2, \ \|v_h\|\sim 1/\sqrt2.
\end{equation}
Now by Proposition \ref{2t6} there exists a sequence of $u_h-w_h\in H^{2}$ that
\begin{gather}
(\Delta -h^{-2}) (u_h-w_h)=o_{L^2}(h^2),\\
(i\pd_n \pm \underline{a} h^{-1}+i\underline{b})(u_h-w_h)=g_h=o_{\Gamma}(h), \ g_h\in H^{1/2}(\Gamma),
\end{gather}
the inhomogeneous terms are the same as those in \eqref{3l3} and \eqref{3l10}. Apply \eqref{2l29} from Proposition \ref{2t4} to see
\begin{equation}\label{3l4}
\|w_h-u_h\|\le Ch^{-1}o(h^2)+o(h)=o(h).
\end{equation}
The equation for $w_n$ is then
\begin{gather}\label{3l25}
\Delta w_h -h^{-2} w_h=o_{H^1}(h)\\
i\pd_n w_h \pm \underline{a} h^{-1}+i\underline{b}w_h=0,\text{ on }\Gamma.
\end{gather}
Now apply the resolvent estimate \eqref{2l32} from Proposition \ref{2t4} to see
\begin{equation}
\|w_h\| =o(h).
\end{equation}
Because of \eqref{3l4} we now have $\|u_h\|=o(h)$, but we know $\|u_h\|\sim h/\sqrt{2}$ beforehand, a contradiction. Hence we have established \eqref{3l9} which further implies the emptiness of $\{i\lambda: \abs{\lambda}\ge \epsilon\}\cap\operatorname{Spec}(A)$. 

6. The contradiction to \eqref{3l8} on a compact non-product cylinder is in general similar. However note that, since we are not on a product cylinder, in Step 5b we could no longer use \eqref{2l32} to get better decay from the better regularity given by the $H^1$-quasimodes. This would lead to a loss of decay, and we have to presume a quasimode of $\abs{\lambda}^{-1}$-better decay,
\begin{equation}
(A\pm i\lambda_n)(u_n, v_n)=o_E(\abs{\lambda_n}^{-3})
\end{equation}
instead of \eqref{3l17}. Note that though we would then gain only $\abs{\lambda}^{-1/2}$ in $\|\sqrt{a}v_n\|_{\Gamma}$ in \eqref{3l19} from the $\abs{\lambda}^{-1}$-better quasimode, it does not seem to prevent the monotonicity argument from working. We eventually arrive at \eqref{3l3} and \eqref{3l10} in Step 5b with new equations
\begin{gather}\label{3l20}
\Delta u_h-h^{-2} u_h=o_{L^2}(h^3)+o_{H^1}(h^2),\\
\pd_n u_h\pm \underline{a}h^{-1} u_h+i\underline{b}u_h=o_{\Gamma}(h^{\frac32})
\end{gather}
Oversimplify the right hand side of \eqref{3l20} by $o_{L^2}(h^2)$ and directly apply \eqref{2l29} from Proposition \ref{2t4} to see
\begin{equation}
\|u_h\|\le Ch^{-1}o(h^2)+o(h^{\frac32})=o(h),
\end{equation}
which is the desired contradiction at the end of Step 5b. The rest of the argument is the same. 
\end{proof}
\begin{remark}\label{3t4}
\begin{enumerate}[wide]
\item Note that by the monotonicity argument to replace the damping functions by those of better regularity, in the boundary datum we lost half of the decay from the presumptive quasimode \eqref{3l17}. This can be seen by comparing \eqref{3l15} to \eqref{3l12}. But this loss happens on the boundary and hence does not impede the sharp polynomial decay. 
\item The efficient use of the $H^1$-part of the quasimode \eqref{3l11} is often seen in the analysis of waves non-uniformly stabilised by interior damping, for example in \cite{aln14} and \cite{bj16}. We remark here that for boundary damping, only in the case of product cylinders we could use the $H^1$-quasimode efficiently. In the case of non-product cylinders, the failure to do so generates a polynomial decay rate presumptively suboptimal by one order, as we observed in the Theorem \ref{0t5}. The ineffective use of the $H^1$-quasimode in this paper is due to the observation that the second derivatives of $u$ have some conormal components at the boundary which do not satisfy another impedance boundary condition. It is noted that in the case of interior damping, the derivatives of $u$ in general satisfy another damped wave equation modulo some $h$-sized error in the 0-th order terms, which is very different from the boundary case. 
\item For the Dirichlet boundary condition \eqref{1l18} and the periodic boundary condition \eqref{1l17}, it suffices to note that the $o_{H^1}$-quasimode in \eqref{3l25} always satisfies the boundary conditions set for $X$ in \eqref{3l27} and \eqref{3l26}.
\end{enumerate}
\end{remark}

We quote \cite[Theorem 2.4]{bt10} to characterise the polynomial decay:
\begin{theorem}[Borichev-Tomilov]\label{3t3}
Let $A$ be a maximal dissipative operator that generates a contraction $C^0$-semigroup in a Hilbert space $X$ and assume that $i\setr\cap \sigma(A)$ is empty. Fix $k>0$, and the following are equivalent:
\begin{enumerate}[wide]
\item There exists $C>0$ such that for any $\abs{\lambda}>0$ one has
\begin{equation}
\left\|(A-i\lambda)^{-1}\right\|_{X\rightarrow X}\le C\abs{\lambda}^k.
\end{equation}
\item There exists $C>0$ such that
\begin{equation}
\left\|e^{tA}A^{-1}\right\|_{X\rightarrow X}\le C t^{-\frac{1}{k}}.
\end{equation}
\end{enumerate}
\end{theorem}

\begin{proof}[Proof of Theorem \ref{0t1} and Theorem \ref{0t6}] The resolvent estimate \eqref{3l9} on the generator $A$ from Proposition \ref{3t2} with Theorem \ref{3t3} of Borichev and Tomilov implies
\begin{equation}
E(u, t)^{\frac12}=\|e^{tA}A^{-1}A(u_0, u_1)\|_E\le Ct^{-\frac12}(\|u_0\|_{H^2}+\|u_1\|_{H^1}),
\end{equation}
as claimed in Theorem \ref{0t1}. Apply Theorem \ref{3t3} again to \eqref{3l8} to establish Theorem \ref{0t5}.
\end{proof}

We now show the sharpness of Theorem \ref{0t1}. We will consider the case in which the cross-section is $\mathbb{B}^d$. The case $d=1$ was investigated in \cite{nis13}. We present the cases when $d\ge 2$ here:
\begin{proposition}[Sharpness]\label{3t5}
Let $\Omega=\mathbb{B}^d_x\times[0,2\pi]_y$, where $\mathbb{B}^d$ is the unit closed ball in dimension $d\ge 2$. Consider the eigenvalue problem 
\begin{gather}\label{3l34}
P_\lambda u(x,y)=(\Delta-\lambda^2)u(x, y)=0, \ (x,y)\in \Omega,\\
i\pd_n u(x,y)+\lambda u(x, y)=0, \ (x,y)\in \mathbb{S}^{d-1}_x\times[0,2\pi]_y,\\
i\pd_n u(x,y)=0, \ (x,y)\in \mathbb{B}^d_x\times\{0,2\pi\}_y,
\end{gather}
Then for $k\in \mathbb{N}$, there exists a sequence of eigenvalues $\lambda_k\in\mathbb C$ such that
\begin{equation}
\lambda_k=\sgn J'_{d/2-1}(\lambda'_0)\sqrt{(\lambda'_0)^2+k^2}-ik^{-1}(\sqrt{(\lambda'_0)^2+k^2})^{-1}+\bigo(k^{-3}),
\end{equation}
where $J_{d/2-1}$ is the Bessel function of parameter $d/2-1$, $\lambda'_0$ is a positive real zero of $J_{d/2-1}$. That is, there exist infinitely many poles in $\{\lambda:0\ge \cim\lambda\ge C^{-1}\abs{\cre{\lambda}}^{-2+\epsilon}, \abs{\cre{\lambda}}\ge C\}$ for any $C,\epsilon>0$. Hence it is impossible for the damped wave equation \eqref{1l1} to be stable at the rate $t^{-1/(2-\epsilon)}$ for any $\epsilon>0$ in those geometric settings. 
\end{proposition}
\begin{proof}
1. It should be noted any eigenvalue $\lambda$ to this problem must satisfy that $\cim(\lambda)\le 0$, as taking the imaginary part of
\begin{equation}\label{3l29}
0=\langle P_\lambda u, u\rangle=\|\nabla u\|^2-\lambda^2\|u\|^2-i\lambda \|u\|_{\mathbb{S}^{d-1}\times[0,2\pi]}^2
\end{equation}
allows us to conclude 
\begin{equation}
0\le \|u\|_{\mathbb{S}^{d-1}\times[0,2\pi]}^2=-2\cim{\lambda}\|u\|^2
\end{equation}
when $\cre{\lambda}\neq0$. When $\cre{\lambda}=0$, the real part of \eqref{3l29} is
\begin{equation}
0=\|\nabla u\|^2+(\cim{\lambda})^2\|u\|^2+\cim{\lambda}\|u\|_{\mathbb{S}^{d-1}\times[0,2\pi]}^2
\end{equation}
which implies $u=0$ is trivial when $\cim{\lambda}>0$. We will use this fact later. 

2. We use the method of separation of variables. Use the spherical coordinates $x=(r,\omega)\in\mathbb{R}_{\ge0}\times \mathbb{S}^{d-1}$ on $\mathbb{B}^{d}$. Assume the rotational symmetry of $u$ and apply the Fourier decomposition with respect to $e_k(y)=\cos(ky)$:
\begin{equation}\label{3l30}
u(r,y)=\sum u_k(r)e_k(y).
\end{equation}
The equation is reduced to the Bessel equation in $r\in [0,1]$ to solve for $u_k\in L^2(0,1)$ such that
\begin{gather}\label{3l38}
(-\pd_r^2-\frac{2}{r}\pd_r-\lambda'^2)u_k(r)=0, \ r\in [0,1],\\
i\pd_r u_k(1)+\lambda u_k(1)=0,
\end{gather}
where $\lambda^2=\lambda'^2+k^2$. 

3. Consider the Bessel function $J_{d/2-1}(z)$, and $\lambda'_0>0$ be a real zero for $J_{d/2-1}(z)$. We will be using the branch cut along $z^{1/2}$ at $\{z: z\in\mathbb{R}_{\ge 0}\}$ and then $z\mapsto z^{1/2}$ is a holomorphic map from $\mathbb{C}\setminus \mathbb{R}_{\ge 0}$ to $\{z:\cim{z}<0\}$. Also use the convention that for $z\in\mathbb{R}_{\ge 0}$ we have $z^{1/2}=(z-i0)^{1/2}$. For large $k$, consider 
\begin{equation}\label{3l31}
F(z)=\frac{J_{d/2-1}(z)}{z(- J_{d/2-1}^2(z)-J_{d/2-1}'(z)^2)^{1/2}}=\frac{1}{k},
\end{equation}
near $z=\lambda_0'>0$ away from $0$. Note that for Bessel functions we have the parameter shifting property
\begin{equation}
J_{d/2}(z)=\frac{d/2-1}{z}J_{d/2-1}(z)-J_{d/2-1}'(z),
\end{equation}
and as $J_{d/2}$ and $J_{d/2-1}$ have distinct zeroes, $J_{d/2-1}'(\lambda'_0)\neq 0$. So $z=\lambda'_0$ is a zero for $F(z)$. Compute
\begin{equation}
F'(\lambda'_0)=J'_{d/2-1}(\lambda'_0)/\lambda'_0(-J'_{d/2-1}(\lambda'_0)^2)^{1/2}=i(\lambda'_0)^{-1}\sgn J_{d/2-1}'(\lambda'_0).
\end{equation}
Thus invoke the implicit function theorem for large $k$ and observe the existence of $\lambda'_k\in\mathbb{C}$ such that
\begin{equation}\label{3l37}
F(\lambda'_k)=\frac{J_{d/2-1}(\lambda'_k)}{\lambda'_k(- J_{d/2-1}^2(\lambda'_k)-J_{d/2-1}'(\lambda'_k)^2)^{1/2}}=\frac{1}{k},
\end{equation}
where
\begin{equation}\label{3l32}
\lambda'_k=\lambda'_0-i(\lambda'_0)\sgn J_{d/2-1}'(\lambda'_0)k^{-1}+\bigo(k^{-2}).
\end{equation}
Thus
\begin{equation}
\lambda_k^2:=(\lambda'_k)^2+k^2=(\lambda'_0)^2+k^2-2i\sgn J'_{d/2-1}(\lambda'_0)k^{-1}+\bigo(k^{-2})
\end{equation}
and
\begin{equation}\label{3l28}
\lambda_k=\sgn J'_{d/2-1}(\lambda'_0)\sqrt{(\lambda'_0)^2+k^2}-ik^{-1}(\sqrt{(\lambda'_0)^2+k^2})^{-1}+\bigo(k^{-3}),
\end{equation}
as claimed in the statement of this proposition. 

4. Now we construct our eigenfunctions. Note $u_k=J_{d/2-1}(\lambda_k' r)$ solves the Helmholtz equation in the interior of $\mathbb{B}^d_x$. Square \eqref{3l37} and we have
\begin{equation}
((\lambda'_k)^2+k^2)J_{d/2-1}^2(\lambda'_k)=-(\lambda'_k)^2J_{d/2-1}'(\lambda'_k)^2,
\end{equation}
which reduces to 
\begin{equation}\label{3l33}
(\lambda_kJ_{d/2-1}(\lambda'_k)+i\lambda'_kJ_{d/2-1}'(\lambda'_k))(\lambda_kJ_{d/2-1}(\lambda'_k)-i\lambda'_kJ_{d/2-1}'(\lambda'_k))=0.
\end{equation}
We will show the first term on the left vanishes. Assume towards contradiction that the second term on the left instead is 0. There are two cases depending on whether $d$ is even or odd:

4a. When $d$ is even, and as 
\begin{equation}
J_s(z)=(z/2)^{s}\sum_{k=0}^\infty (-1)^k\frac{(z/2)^{2k}}{k!\Gamma(s+k+1)},
\end{equation}
we know $\overline{J_{d/2-1}(z)}=J_{d/2-1}(\bar{z})$. Thus
\begin{equation}
\bar\lambda_kJ_{d/2-1}(\overline{\lambda_k'})+i\overline{\lambda_k'}J_{d/2-1}'(\overline{\lambda_k'})=0
\end{equation}
and $J_{d/2-1}(\overline{\lambda_k'} r)$ solves the equation \eqref{3l38} with $\cim{\overline{\lambda_k}}>0$. This contradicts our observation in Step 1. Hence we must have
\begin{equation}
\lambda_kJ_{d/2-1}(\lambda_k')+i\lambda_k'J_{d/2-1}'(\lambda_k')=0,
\end{equation}
and thus $J_{d/2-1}(\lambda_k' r)$ solves the equation \eqref{3l38}. 

4b. When $d$ is odd, note $\overline{z^{1/2}}=-\bar{z}^{1/2}$ and hence $\overline{J_{d/2-1}(z)}=-J_{d/2-1}(\bar{z})$. A similar argument tells us 
\begin{equation}
-\bar\lambda_kJ_{d/2-1}(\overline{\lambda_k'})-i\overline{\lambda_k'}J_{d/2-1}'(\overline{\lambda_k'})=0
\end{equation}
and the rest follows. 

5. We now show that the damped wave equation \eqref{1l1} cannot be stable at the rate $t^{-1/(2-\epsilon)}$ for any $\epsilon>0$. Firstly observe from the asymptotics \eqref{3l28}  that for any $C>0$ there are infinitely many complex eigenvalues in 
\begin{equation}
\{\lambda:0\ge \cim\lambda\ge C^{-1}\abs{\cre{\lambda}}^{-2+\epsilon}, \abs{\cre{\lambda}}\ge C\}.
\end{equation}
Write $\lambda_k=\mu_k-i\sigma_k$ with
\begin{equation}
\mu_k\sim k\rightarrow \infty, \ \sigma_k\sim k^{-2}\rightarrow 0^+
\end{equation}
up to renumbering. Let $v_k=i\lambda_k u_k$ and then from \eqref{3l34} we have
\begin{equation}
(A-i\mu_k)(u_k,v_k)=\sigma_k(u_k,v_k)\sim \mu_k^{-2}(u_k,v_k)=o(\mu_k^{-2+\epsilon})(u_k,v_k)
\end{equation}
and hence
\begin{equation}
\|(A-i\mu_k)\|_{E\rightarrow E}=o(\mu_k^{-2+\epsilon})
\end{equation}
and by Theorem \ref{3t3}, it is impossible that $\|e^{tA}A^{-1}\|_{E\rightarrow E}=\bigo(t^{-1/(2-\epsilon)})$. 
\end{proof}

\section{Resolvent Estimates and Evolution on Infinite Cylinders}
We now consider the corresponding interior impedance problem to the damped wave equation \eqref{1l10} to \eqref{1l11} on an infinite non-product cylinder $\Omega$:
\begin{gather}
\label{4l3}
(\Delta-\lambda^2)u=f\text{ on }\Omega,\\
\label{4l4}
i\pd_n u+a\lambda \gamma u+b \gamma u=g \text{ on } \Gamma.
\end{gather}
Firstly note that as $\pd\Omega=\Gamma$ in the case of infinite cylinders, there are no $\Gamma_\pm$ compared to the case of compact cylinders. We start by showing we have a resolvent estimate similar to Proposition \ref{2t4}.
\begin{proposition}\label{4t1}
Consider a solution $u\in H^2(\Omega)$ to the interior impedance problem \eqref{4l3}. Assume the damping assumption \ref{0ad}. We have the following estimates:
\begin{enumerate}[wide]
\item On an infinite non-product cylinder $\Omega$ defined in Definition \ref{0gin}, assume $f\in L^2(\Omega)$ and $g\in L^2(\Gamma)$. Then for any $\epsilon>0$, we have some $C>0$ independent of $\lambda, u, f, g$ such that for any $\abs{\lambda}\ge\epsilon$,
\begin{equation}\label{4l10}
\|u\|_{L^2(\Omega)}^2+(1+\lambda^2)^{-1}\|\nabla u\|_{L^2(\Omega)}^2\le C(1+\lambda^{2})\|f\|_{L^2(\Omega)}^2+C\|g\|_{L^2(\Gamma)}^2.
\end{equation}

\item On an infinite product cylinder $\Omega$ defined in Definition \ref{0gp}, assume further that $a,b,\pd_y a,\pd_y b\in C^{0,1/2+\delta}(\Gamma)$ for some $\delta>0$, $f\in H^1(\Omega)$, $g\in H^{3/2}(\Gamma)$. Then for any $\epsilon>0$, we have some $C>0$ independent of $\lambda, u, f, g$ such that for any $\abs{\lambda}\ge\epsilon$, 
\begin{equation}
\|u\|_{L^2(\Omega)}^2+(1+\lambda^2)^{-1}\|\nabla u\|_{L^2(\Omega)}^2\le C\|f\|_{H^1(\Omega)}^2+C\|g\|_{H^1(\Gamma)}^2.
\end{equation}

\item We could allow $\epsilon=0$ in (1) and (2), if one further assumes
\begin{equation}
0< b_0\le b(x,y)\le b_1 
\end{equation}
almost everywhere on $\Gamma$.

\end{enumerate}
\end{proposition}

\begin{proof}
It suffices to observe that Proposition \ref{2t1} and Corollary \ref{2t3} still hold with similar proofs in this case, except we no longer have $\Gamma_\pm$ here, which is easier. 
\end{proof}

\begin{proposition}[Classical Elliptic Regularity on Infinite Cylinders]
Let $\Omega$ be an infinite non-product cylinder $\Omega$ defined in Definition \ref{0gin}. Let $f\in L^2(\Omega)$, $g\in H^{1/2}(\Gamma)$ and $a, b\in C^{0,1/2+\delta}(\Gamma)$ for any $\delta>0$. When $b\not\equiv 0$ identically, for any $\cre{z}\ge 0$ there exists a unique weak solution $u\in H^{1}(\Omega)$ to the equation
\begin{gather}\label{4l18}
(\Delta +z^2)u=f,\text{ on }\Omega,\\
\pd_n u+a z u+bu=G,\text{ on }\Gamma. 
\end{gather}
When $b\equiv 0$, the same statement holds for $\{\cre{z}\ge 0, z\neq 0\}$; and at $z=0$, there exists a weak solution in $H^1$ if and only if 
\begin{equation}
\langle f, 1\rangle+\langle G,1\rangle_{\Gamma}=0,
\end{equation}
and such solution is instead unique modulo constant functions. Moreover, any such solution $u\in H^1(\Omega)$ must be in $H^{2}(\Omega)$. 
\end{proposition}
\begin{proof}
1. Regularity: we claim if $u\in H^1(\Omega)$ solves the equation \eqref{4l18}, we have $u\in H^2(\Omega)$: indeed 
\begin{equation}\label{4l19}
\|u\|_{H^2(\Omega)}\le C\left(\|u\|_{H^1(\Omega)}+\|f\|_{L^2(\Omega)}+\|G\|_{H^{1/2}(\Gamma)}\right).
\end{equation}
Consider the $C^{1,1}$-scaling function $\rho: \mathbb{S}^{d-1}_{w'}\times \mathbb{R}_{y'}\rightarrow \mathbb{R}_{\ge c_0'}$. Note $\Omega\subset \mathbb{R}^d_x\times \mathbb{R}_y$ and let $\Omega'=\mathbb{B}^{d}_{x'}\times \mathbb{R}_{y'}$. We will use the spherical coordinates in $(r,w)$ on $\mathbb{R}^d_{x}$ and $(r',w')$ on $\mathbb{R}^d_{x'}$, both of which are in $\mathbb{R}_{\ge 0}\times\mathbb{S}^{d-1}$. Consider the $C_b^{1,1}$-diffeomorphism $\Phi:(\Omega', g_{Euc})\rightarrow (\Omega,g_{Euc})$ given by 
\begin{equation}
(r, w, y)=\Phi(r',w',y')=(r'\rho(w', y'),w',y').
\end{equation}
Denote the global bound of the scaling functions by
\begin{equation}
\rho_1=\sup\{\abs{\rho(w', y')}, \abs{\rho(w', y')}^{-1}, \abs{\nabla \rho(w',y')}: (w', y')\in \mathbb{S}^{d-1}_{w'}\times\mathbb{R}_{y'}\}, 
\end{equation}
and throughout Step 1, the constants $C(\rho_1)$ in the estimates that we will show will be dependent on $\rho_1,a,b,z$ only. Compute the pushforward induced by $\Phi$:
\begin{equation}
d\Phi: (\pd_{r'}, X\cdot\nabla_{w'}, \pd_{y'})\mapsto (\rho\pd_{r}, X\cdot\nabla_w+r'(X\cdot\nabla_{w'}\rho)\pd_r,\pd_y+r'(\pd_{y'}\rho)\pd_r)
\end{equation}
for some choice of coordinates $w=w'$ on $S^{d-1}$ and any $X\cdot\nabla_{w'}\in\operatorname{Vect}(\mathbb{S}^{d-1})$. It is essential to observe that $d\Phi$ is bounded from above and below:
\begin{equation}
\|d\Phi\|_{T_{p'}\Omega'\rightarrow T_{\Phi(p')}\Omega}\le C(\rho_1),\ \|d\Phi^{-1}\|_{T_{\Phi(p')}\Omega\rightarrow T_{p'}\Omega'}\le C(\rho_1)
\end{equation}
at any $p'\in \Omega'$. The equation \eqref{4l18} on $\Omega'$ now reads
\begin{gather}
(\Delta_{g'}+z^2)\Phi^*u=\Phi^*f,\text{ on }\Omega',\\
\pd_{n'} \Phi^*u=\Phi^*G',\text{ on }\Gamma',\\
G'=G-azu-bu,
\end{gather}
where $g'=\Phi^*g_{Euc}$, and $\pd_{n'}=d\Phi^{-1}(\pd_{n})$ on $\Gamma'=\pd\Omega'$. Now pick a smooth cutoff function $\chi(y')$ identically 1 on $[-1,1]$ and supported in $[-2,2]$, and a smooth cutoff function $\tilde\chi(y')$ identically 1 on $\supp{\chi}$ and supported in $[-3,3]$. Consider the translation operator $\tau_s: (r', w', y')\mapsto (r', w', y'+s)$. Then the equation for $\chi \tau^*_s u$ is
\begin{gather}
Q_{s}\chi\tau^*_s \Phi^*u=\chi\tau^*_s \Phi^*f+[\chi, \pd_{y'}^2]\tau^*_s \Phi^*u,\text{ on }\Omega''=\mathbb{B}_{x'}\times[-3,3]_{y'},\\
\pd_{n'} \chi \tau^*_s \Phi^*u=\chi \tau^*_s \Phi^*G,\text{ on }\Gamma''=\mathbb{S}_{x'}\times[-3,3]_{y'},\\
Q_{s}=\tau^*_s \Delta_{g'}+z^2,
\end{gather}
on $(\Omega'',g_{Euc})$.
% We claim $\tau_{s*}:T^*_p(\Omega')\rightarrow T^*_{\tau_s(p)}(\Omega')$ is a map bounded from below, and the lower bound is independent of $s, p$. Indeed, it suffices to observe that the map $\psi_{-s}:=\Phi\circ \tau_{-s} \circ \Phi^{-1}$ is given by
% \begin{equation}
% \psi_{-s}(r, w, y)=\left(\rho(w, y-s)\rho^{-1}(w, y)r, w, y-s\right),
% \end{equation}
% and is bounded from above:
% \begin{equation}
% \|d\psi_{-s}\|_{T_{\Phi(p)}(\Omega)\rightarrow T_{\tau_{-s}\Phi(p)}(\Omega)}\le C(\rho_0,\rho_1, \rho_2).
% \end{equation}
% Since $d\Phi$ and $d\Phi^{-1}$ are both bounded from above and below, $d\tau_{-s}$ is bounded from above and $d\tau_{s}$ is bounded from below, which concludes the claim. 
% The classical principal symbol of $Q_s$ is elliptic uniformly in parameter $s$:
% \begin{multline}
% \sigma(Q_s)(p, \xi)=\sigma(\Delta_g)(\tau_s(p), \tau_{s*}(\xi))=\abs{\Phi_*(\tau_{s}(p),\tau_{s*}(\xi)}_{T^*_{\tau_s(p)}(\Omega)}^2\\
% \ge C(\rho)\abs{(\tau_{s}(p),\tau_{s*}(\xi)}_{T^*_{\tau_s(p)}(\Omega')}^2\ge C(\rho)\abs{(p,\xi)}_{T^*_p\Omega'},
% \end{multline}
% where $C(\rho)$ only depends on $\rho$, because $\Phi_*$ and $\tau_{s*}$ are bounded from below by such constants. 
Since the first derivatives of $g'$ are bounded by $C(\rho_1)$, it could be checked that the quadratic form associated to $Q_s$,
\begin{multline}
B_s(u,v)=\langle (\Delta_{g'}+z^2) u, v\rangle+\langle \pd_{n'}u, v\rangle_{\Gamma''} 
=\sum_{i,j}\langle g'^{ij}\sqrt{g'}\pd_j u, \pd_i  (v/\sqrt{g'})\rangle\\
+z^2\langle u, v\rangle =\langle \nabla_{g'}u, \nabla_{g'} v\rangle +\sum_{i,j}\langle (\pd_i\sqrt{g'})g'^{ij}/\sqrt{g'}\pd_j u, v\rangle+z^2\langle u, v\rangle
\end{multline}
is coercive on $H^1(\Omega'')$ according to \cite[Theorem 4.6, 4.7]{mcl00}, that is,
\begin{equation}
B_s(u,u)\ge C(\rho_1)\|u\|_{H^1}^2-C(\rho_1)\|u\|_{L^2}^2
\end{equation}
for each $u\in H^1(\Omega'')$. Note that this coercive constant is independent of the parameter $s$. Now estimate:
\begin{gather}
\|[\chi, \pd_{y'}^2]\tau^*_s \Phi^*u \|\le C\|\tilde\chi \tau^*_s \Phi^*u\|_{H^1},\\
\|\chi\tau^*_s \Phi^*G\|_{H^{1/2}(\Gamma'')}\le C(\rho_1)(\|\tilde\chi \tau^*_{s}\Phi^*g\|_{H^{1/2}(\Gamma'')}+\|\tilde\chi \tau^*_{s}\Phi^*u\|_{H^1}).
\end{gather}
Invoke the standard elliptic regularity result on $\Omega''$, as one can carefully trace down the coefficient dependency in \cite[Theorem 4.18]{mcl00}, to see
\begin{equation}
\|\chi\tau^*_s\Phi^*u\|_{H^2}\le C(\rho_1)\left(\|\tilde\chi\tau^*_s\Phi^*u\|_{H^1}+\|\chi\tau^*_s \Phi^*f\|+\|\tilde\chi \tau^*_{s}\Phi^*g\|_{H^{1/2}(\Gamma'')}\right),
\end{equation}
where the constant is independent of $s$. Sum up those estimates at different $s\in\mathbb{Z}$ to see 
\begin{equation}
\|\Phi^*u\|_{H^2(\Omega')}\le C\left(\|\Phi^*u\|_{H^1(\Omega')}+\|\Phi^*f\|_{L^2(\Omega')}+\|\Phi^*G\|_{H^{1/2}(\Gamma')}\right).
\end{equation}
Since $\Phi:\Omega'\rightarrow\Omega$ is a $C^{1,1}_b$-diffeomorphism, we know $H^t(\Omega')=\Phi^*H^t(\Omega)$ for any $t\in [0,2]$ and they have equivalent norms. This implies \eqref{4l19}.

2. The rest of the argument is the same as in Proposition \ref{2t6}.
\end{proof}

From here we always assume
\begin{equation}\label{4l6}
0<b_0\le b(p)\le b_1, \ p\in\Gamma.
\end{equation}
Consider the Hilbert space $X=H^2(\Omega)\oplus H^1(\Omega)$. Let the generator be 
\begin{equation}
A=\begin{pmatrix}
0 & \id\\
-\Delta & 0
\end{pmatrix}: D(A)\rightarrow X,
\end{equation}
where $D(A)$ comprises those $(u,v)\in X$ such that
\begin{equation}
\pd_n u(p)+a(p)v(p)+b(p)u(p)=0, \ p\in\Gamma.
\end{equation}
Equip $X$ with the norm of $H^1\oplus L^2$. Consider the energy seminorm
\begin{equation}
\|(u,v)\|_E^2=\|\nabla u\|_{L^2(\Omega)}^2+\|v\|_{L^2(\Omega)}^2+\|\sqrt{b}u\|_{L^2(\Gamma)}^2.
\end{equation}
\begin{proposition}
On an infinite non-product cylinder $\Omega$ defined in Definition \ref{0gin}, impose the damping assumption \ref{0ad} and assume $a,b\in C^{0,1/2+\delta}(\Gamma)$, $b(p)\ge b_0>0$ on $\Gamma$, then the energy seminorm $E$ is a norm on $X$ and is equivalent to $H^1\oplus L^2$. Furthermore, the generator $A$ is a maximally dissipative with respect to the energy norm $E$. Hence $A$ generates a strongly continuous semigroup $e^{tA}$ which is a contraction.
\end{proposition}
\begin{proof}
Similar to the proof of Proposition \ref{3t1}. Note that in this case we have $b\ge b_0>0$ on $\Gamma$. 
\end{proof}

\begin{proposition}
On a infinite non-product cylinder $\Omega$ defined in Definition \ref{0gin}, impose the damping assumption \ref{0ad}, and assume $b(x,y)\ge b_0>0$ on $\Gamma$, then the generator $A:D(A)\rightarrow X$ has no purely imaginary spectrum. Furthermore we have
\begin{equation}\label{4l5}
\|(A+i\lambda)^{-1}\|_{E\rightarrow E}\le C \langle \lambda\rangle^3.
\end{equation}
Assume further that $\Omega$ is an infinite product cylinder defined in Definition \ref{0gp}, then we have
\begin{equation}\label{4l7}
\|(A+i\lambda)^{-1}\|_{E\rightarrow E}\le C\langle \lambda\rangle^2,
\end{equation}
which improves \eqref{4l5}. 
\end{proposition}
\begin{proof}
We will follow a similar set-up as in the proof of Proposition \ref{3t2}. Note that $\Omega$ is no longer compact in this geometric setting, whence the Step 1 and 2 there do not work and the spectrum is not necessarily discrete. We will deal with the zero-frequency quasimodes in a more direct manner.

1. We first show \eqref{4l7} when $\Omega$ is an infinite product cylinder. Assume \eqref{4l7} is not true, and there are $\lambda_n$ and $(u_n, v_n)\in D(A)$ with
\begin{equation}\label{4l2}
\|(u_n, v_n)\|_E^2=\|\nabla u_n\|^2+\|v_n\|^2+\|b^{1/2}u_n\|^2_{\Gamma}\equiv 1
\end{equation}
such that $(A+i\lambda_n)(u_n, v_n)=o_E(\langle \lambda_n\rangle^{-2})=o_X(\langle \lambda_n\rangle^{-2})$, which is
\begin{gather}
i\lambda_n u_n+v_n=o_{H^1}(\langle \lambda_n\rangle^{-2}),\\
-\Delta u_n+i\lambda_n v_n=o_{L^2}(\langle \lambda_n\rangle^{-2}),
%b^\frac12 (\pm i\lambda_n u_n+v_n)=o_{L^2(\Gamma)}(\abs{\lambda_n}^{-2})
\end{gather}
which could be written as
\begin{gather}
v_n=- i\lambda_n u_n+o_{H^1}(\langle \lambda_n\rangle^{-2}),\\
\Delta u_n-\lambda_n^2 u_n=o_{L^2}(\langle \lambda_n\rangle^{-2})+o_{H^1}(\langle \lambda_n\rangle^{-1}).
\end{gather}
Replace the damping functions as in Step 4 of the proof of Proposition \ref{3t2} to see
\begin{gather}\label{4l8}
\Delta u_n-\lambda_n^2 u_n=o_{L^2}(\langle \lambda_n \rangle^{-2})+o_{H^1}(\langle \lambda_n \rangle^{-1}),\\
\label{4l9}
i\pd_n u_n\pm \underline{a}\lambda_n u_n+i\underline{b}u_n=g_n=o_{\Gamma}(\langle \lambda_n \rangle^{-1}),\ g_n\in H^{1/2}(\Gamma),
\end{gather}
where $\underline{a}(p),\underline{b}(p)\in C^{0,1/2+\delta}(\Gamma)\cap L^\infty(\Gamma)$ satisfy that
\begin{gather}
0< a_0\le \underline{a}(p)\le a(p), \ p\in \Gamma,\\
0\le \underline{b}(p)\le b(p), \ p\in \Gamma,\\
\pd_y\underline{a}(x,y),\pd_y\underline{b}(x,y)\in C^{0,1/2+\delta}(\Gamma)\cap L^\infty(\Gamma).
\end{gather}
The high and low frequency regimes are similar to those in Step 5a and 5b of the proof of Proposition \ref{3t2}, with a monotonicity argument, as in those regimes $\langle \lambda\rangle$ is comparable to $\abs{\lambda}$. We are left with the zero frequency regime.

2. The zero frequency regime: assume $\lambda=\abs{\lambda_n}\rightarrow 0$. The system \eqref{4l8} to \eqref{4l9} now reads
\begin{gather}
\label{4l1}
\Delta u_\lambda-\lambda^2 u_\lambda=o_{L^2}(1),\\
(i\pd_n + \underline{a} \lambda+i\underline{b})u_\lambda=g_\lambda=o_{\Gamma}(1), \ g_\lambda\in H^{1/2}(\Gamma).
\end{gather}
As $b\ge b_0>0$, the estimate \eqref{4l10} from Proposition \ref{4t1} with $\epsilon=0$ implies
\begin{equation}
\|u_\lambda\|^2+\|\nabla u_\lambda\|^2=o(1),
\end{equation}
implying
\begin{equation}
\|v_\lambda\|^2=o(1), \, \|u_\lambda\|_{\Gamma}^2\le C(\|u_\lambda\|^2+\|\nabla u_\lambda\|^2)=o(1),
\end{equation}
which contradicts \eqref{4l2}. This concludes the proof of \eqref{4l7}.

3. The proof of \eqref{4l5} is similar. 
\end{proof}

\begin{proof}[Proof of Theorem \ref{0t4} and Theorem \ref{0t5}]
Apply Theorem \ref{3t3} of Borichev and Tomilov to \eqref{4l5} and \eqref{4l7} from Proposition \ref{4t1} to conclude the proof. 
\end{proof}

\begin{proposition}[Sharpness]\label{4t2}
Let $\Omega=\mathbb{B}^d_x\times \mathbb{R}_y$, where $\mathbb{B}^d$ is the unit closed ball in dimension $d\ge 2$. Consider the eigenvalue problem
\begin{gather}
P_\lambda u(x,y)=(\Delta-\lambda^2)u(x, y)=0, \ x\in \Omega,\\
i\pd_n u(x,y)+\lambda u(x, y)+i u(x,y)=0, \ (x,y)\in \mathbb{S}^{d-1}_x\times \mathbb{R}_y,
\end{gather}
in the distributional sense. Then there exists some $C>0$ such that the spectrum contains the set $\{\lambda_{\eta_0}:\eta_0\in \mathbb{R}, \abs{\eta_0}\ge C\}$ in which the complex eigenvalues
\begin{equation}
\lambda_{\eta_0}=\sgn J'_{d/2-1}(\lambda'_0)\sqrt{(\lambda'_0)^2+\eta_0^2}-i\eta_0^{-1}(\sqrt{(\lambda'_0)^2+\eta_0^2})^{-1}+\bigo(\eta_0^{-3})
\end{equation}
correspond to some generalised eigenfunctions in $H^{-1-\epsilon}(\Omega)$ for any $\epsilon>0$, where $J_{d/2-1}$ is the Bessel function of parameter $d/2-1$, $\lambda'_0$ is a positive real zero of $J_{d/2-1}$. Moreover there exist $o_E(\cre(\lambda_{\eta_0})^{-2+\epsilon})$-quasimodes, independent of $\epsilon>0$, of the semigroup generator $A$ at $i\cre(\lambda_{\eta_0})$ as $\cre(\lambda_{\eta_0})\rightarrow \infty$. Hence it is impossible for the damped wave equation \eqref{1l10} to be stable at the rate $t^{-1/(2-\epsilon)}$ for any $\epsilon>0$ in those geometric settings.
\end{proposition}

\begin{proof}Although the proof will be similar to that of Proposition \ref{3t5}, note that now the domain is not bounded in $y$ and we would construct quasimodes instead of finding exact eigenfunctions.

1. We use the Fourier transform on $\mathbb{R}_y$ to turn position $y$ into momentum $\eta$. That is, instead of \eqref{3l30} we look at $\hat u (x, \eta)$ and the system on $\mathbb{B}^d$. We will look for $w_\eta(x)$ being the solution to
\begin{gather}\label{4l11}
(\Delta_x-\lambda'^2)w_{\eta}(x)=0, \ x\in\mathbb{B}^d,\\
\label{4l17}
i\pd_n w_{\eta}(x)+\lambda w_{\eta}(x)+i w_{\eta}(x)=0, \ x\in\mathbb{S}^{d-1},
\end{gather}
where we define $\lambda^2=\lambda'^2+\eta^2$. Let $J=J_{d/2-1}$. Instead of \eqref{3l31} we consider for large $\eta$ that
\begin{equation}
F(z)=\frac{J(z)}{\left(-(z^2+1)J(z)^2-z^2J'(z)^2-2zJ(z)J'(z)\right)^{\frac{1}2}}=\frac{1}\eta
\end{equation}
near $z=\lambda_0'>0$ where $\lambda_0'$ is a real positive zero of $J(z)$. Note that in a sufficiently small complex neighbourhood of $\lambda_0'$ what is inside the square root is close to $-\lambda_0'^2J'(\lambda_0')^2<0$ and is away from the branch cut $\mathbb{R}_{\ge 0}$. Compute
\begin{equation}
F'(\lambda_0')=i(\lambda'_0)^{-1}\sgn J_{d/2-1}'(\lambda'_0)
\end{equation}
and invoke the implicit function theorem for large $\eta$ to see the existence of $\lambda_{\eta}'\in\mathbb{C}$ such that
\begin{equation}\label{4l12}
F(\lambda_\eta')=\frac{J(\lambda_\eta')}{\left(-(\lambda_\eta'^2+1)J(\lambda_\eta')^2-\lambda_\eta'^2J'(\lambda_\eta')^2-2\lambda_\eta'J(\lambda_\eta')J'(\lambda_\eta')\right)^{\frac{1}2}}=\frac{1}\eta
\end{equation}
where we have the same asymptotics to \eqref{3l32} that
\begin{equation}
\lambda_\eta'=\lambda_0'-i(\lambda_0')\sgn J'(\lambda_0')\eta^{-1}+\bigo(\eta^{-2})
\end{equation}
and to \eqref{3l28} that
\begin{equation}\label{4l13}
\lambda_\eta=\sqrt{\lambda_\eta'^2+\eta^2}=\sgn J'(\lambda'_0)\sqrt{\lambda_0'^2+\eta^2}-i\eta^{-1}(\sqrt{\lambda_0'^2+\eta^2})^{-1}+\bigo(\eta^{-3}),
\end{equation}
while $\lambda'_\eta$ and $\lambda_\eta$ depend $C^1$-continuously on $\eta$. Now set $w_\eta(x)=J(\lambda_\eta' \abs{x})$ that solves the system \eqref{4l11} using the argument similar to that in Step 4 of the proof of Proposition \ref{3t5}. Indeed, \eqref{4l12} implies
\begin{equation}
\left(\lambda_\eta J(\lambda_\eta')+i(J(\lambda_\eta')+\lambda_\eta'J'(\lambda_\eta'))\right)\left(\lambda_\eta J(\lambda_\eta')-i(J(\lambda_\eta')+\lambda_\eta'J'(\lambda_\eta'))\right)=0
\end{equation}
which is similar to \eqref{3l33} and the rest follows. Note that $w_\eta(x)$ is not in $L^2(\mathbb{B}^d\times\mathbb{R}_\eta)$ but only in $\langle \eta\rangle^{1+\epsilon}L^2$, or in $L^2_{loc}$, as $w_\eta(x)\sim J(\lambda_0 \abs{x})$, which looks like a constant in $\eta\rightarrow \infty$. 

2. Now we localise the $L^2_{loc}$-solutions $w_\eta$ to \eqref{4l11} in the frequencies to get $L^2$-quasimodes. Let $\chi(\eta)\in \fsccinf(\mathbb{R})$ be a non-negative cutoff function which is supported in $[-1,1]$ and is identically 1 on $[-1/2, 1/2]$. Let the frequency-space cutoff near a fixed frequency $\eta_0\in\mathbb{R}$ sufficiently away from 0 be 
\begin{equation}
\chi_{\eta_0}(\eta)=\chi\left(\eta_0^2(\eta-\eta_0)\right),
\end{equation}
which is supported inside $[\eta_0-\eta_0^{-2}, \eta_0+\eta_0^{-2}]$, an $\eta_0^{-2}$-neighbourhood of $\eta_0$. We will use them as Fourier multipliers. Also note that from the asymptotics \eqref{4l13},
\begin{equation}\label{4l14}
\lambda_\eta^2-\lambda_{\eta_0}^2=(\eta-\eta_0)(\eta+\eta_0)+\bigo(\eta_0^{-1})=\bigo(\eta_0^{-1})
\end{equation}
on $\supp\chi_{\eta_0}$. Let
\begin{equation}\label{4l15}
u_{\eta_0}(x,y)=\chi_{\eta_0}(D_y)(\mathcal{F}_{y\leftarrow \eta}^{-1}w_\eta)(x, y)=\mathcal{F}_{y\leftarrow \eta}^{-1}\left(\chi_{\eta_0}(\eta)w_\eta(x)\right).
\end{equation}
Note that $u_{\eta_0}$ is in $H^1(\mathbb{B}^d\times\mathbb{R}_y)$, since $\chi_{\eta_0}w_\eta$ is $C^1_c$ and in $\langle \eta\rangle^{-1}L^2(\mathbb{B}^d\times\mathbb{R}_\eta)$. Then from \eqref{4l14} we have
\begin{multline}
\mathcal{F}_{y\leftarrow \eta}^{-1}(\lambda_\eta^2\chi_{\eta_0}w_\eta)=\lambda_{\eta_0}^2\mathcal{F}_{y\leftarrow \eta}^{-1}(\chi_{\eta_0}w_\eta)+\mathcal{F}_{y\leftarrow \eta}^{-1}(\bigo(\eta_0^{-1})\chi_{\eta_0}w_\eta)\\
=\lambda_{\eta_0}^2u_{\eta_0}+\bigo_{L^2}(\eta_0^{-1}u_{\eta_0}),
\end{multline}
since by using the Plancherel theorem twice we have
\begin{equation}
\|\mathcal{F}_{y\leftarrow \eta}^{-1}(\bigo(\eta_0^{-1})\chi_{\eta_0}w_\eta)\|=\|\bigo(\eta_0^{-1})\chi_{\eta_0}w_\eta\|=\bigo(\eta_0^{-1})\|u_{\eta_0}\|.
\end{equation}
Now observe for fixed $\eta_0$ large, we have from \eqref{4l15} that
\begin{multline}\label{4l16}
(\Delta-\lambda_{\eta_0}^2)u_{\eta_0}=\mathcal{F}_{y\leftarrow \eta}^{-1}\left((\Delta_x-(\lambda_{\eta}^2-\eta^2))\chi_{\eta_0}(\eta)w_\eta\right)+\bigo_{L^2}(\eta_0^{-1}u_{\eta_0})\\
=\mathcal{F}_{y\leftarrow \eta}^{-1}\left(\chi_{\eta_0}(\eta)(\Delta_x-(\lambda_{\eta}^2-\eta^2))w_\eta\right)+\bigo_{L^2}(\eta_0^{-1}u_{\eta_0})=\bigo_{L^2}(\eta_0^{-1}u_{\eta_0}),
\end{multline}
the term undergoing the inverse Fourier transform of which vanishes since $w_\eta$ solves \eqref{4l11} and meets the boundary condition \eqref{4l17}. Now write $\lambda_{\eta_0}=\mu_{\eta_0}-i\sigma_{\eta_0}$ with
\begin{equation}
\mu_{\eta_0}\sim\eta_0\rightarrow \infty, \ \sigma_{\eta_0}\sim \eta_0^{-2}\rightarrow 0^+
\end{equation}
up to renumbering. Let $v_{\eta_0}=i\lambda_{\eta_0}u_{\eta_0}$ and from \eqref{4l16} we have
\begin{multline}
(A-i\mu_{\eta_0})(u_{\eta_0},v_{\eta_0})=\sigma_{\eta_0}(u_{\eta_0},v_{\eta_0})+(0,\bigo_{L^2}(\eta_0^{-1}u_{\eta_0}))\sim \mu_{\eta_0}^{-2}(u_{\eta_0},v_{\eta_0})\\
+(0, \bigo_{L^2}(\mu_{\eta_0}^{-1}\eta_0^{-1}v_{\eta_0}))=o_{E}(\mu_{\eta_0}^{-2+\epsilon}(u_{\eta_0},v_{\eta_0})),
\end{multline}
whence
\begin{equation}
\|A-i\mu_{\eta_0}\|_{E\rightarrow E}=o(\mu_{\eta_0}^{-2+\epsilon})
\end{equation}
as $\mu_{\eta_0}\rightarrow \infty$. Invoke Theorem \ref{3t3} of Borichev and Tomilov to conclude the proof. 
\end{proof}

%\nocite{*}
\bibliography{bib}
\bibliographystyle{amsalpha}

\end{document}